\documentclass{amsart}
\usepackage{amsmath}
\usepackage{amsfonts, mathrsfs}
\usepackage{amsthm, upref}
\usepackage{graphicx}
\usepackage[usenames, dvipsnames]{color}
\usepackage{appendix}

\usepackage{tikz}
\usepackage{enumerate}
\usepackage[tmargin=1.1in,bmargin=1.0in,lmargin=1.0in,rmargin=1.4in]{geometry}
\usepackage{hyperref}
\usepackage[all]{xy}

\newcommand{\eq}{\begin{equation}}
\newcommand{\en}{\end{equation}}

\newcommand{\norm}[1]{\left\lVert #1 \right\rVert}
\newcommand{\abs}[1]{\left\lvert #1 \right\rvert}

\newcommand{\iprod}[1]{\left\langle #1 \right\rangle }

  \newcommand{\eps}{\epsilon}
  \newcommand{\ea}{\begin{eqnarray*}}
  \newcommand{\ee}{\end{eqnarray*}}
  \newcommand{\nea}{\begin{eqnarray}}
  \newcommand{\nee}{\end{eqnarray}}

  \newcommand{\Aa}{\mathcal{A}}

  \newcommand{\Gg}{\mathcal{G}}
  \newcommand{\Ww}{\mathcal{W}}
  \newcommand{\Cc}{\mathcal{C}}
  \newcommand{\C}{\mathbf{C}}
  \newcommand{\Y}{\mathbf{Y}}

  \newcommand{\RR}{\mathbb{R}}

  \newcommand{\NN}{\mathbb{N}}
  \newcommand{\PP}{\mathbb{P}}
  \newcommand{\ZZ}{\mathbb{Z}}
  
  \newcommand{\limn}{\lim_{n\to\infty}}

  \newcommand{\var}{\mathop{\mathrm{var}}\nolimits}
  \renewcommand{\subset}{\subseteq}

  \newcommand{\E}{\mathbf{E}}
  \newcommand{\Poi}{\mathrm{Poi}}
  \newcommand{\eqd}{\,{\buildrel \mathcal{L} \over =}\,} 
  \newcommand{\toL}{\,{\buildrel \mathcal{L} \over \longrightarrow}\,} 
  \newcommand{\toPr}{\,{\buildrel pr \over \longrightarrow}\,} 
  \newcommand{\dtv}{d_{TV}}
  \renewcommand{\P}{\mathbf{P}}
  \renewcommand{\var}{\mathbf{Var}}

\newcommand{\weight}{\underline{\omega}}
\newcommand{\ltwo}{\mathbf{L}^2}
\newcommand{\Ent}{\text{Ent}}

\newcommand{\expect}{\mathbf{E}}
\newcommand{\expec}{\mathbf{E}}
\newcommand{\R}{\mathbb{R}}
\newcommand{\prob}{\mathbf{P}}
\newcommand{\one}{\mathbf{1}}
\newcommand{\oneL}[2]{L_{[#1,#2]}}

\begin{document}

\theoremstyle{plain}
\newtheorem{thm}{Theorem}
\newtheorem{lemma}[thm]{Lemma}
\newtheorem{prop}[thm]{Proposition}
\newtheorem{cor}[thm]{Corollary}

\theoremstyle{definition}
\newtheorem{defn}[thm]{Definition}
\newtheorem{asmp}{Assumption}
\newtheorem{notn}{Notation}
\newtheorem{prb}{Problem}

\theoremstyle{remark}
\newtheorem{rmk}[thm]{Remark}
\newtheorem{exm}{Example}
\newtheorem{clm}{Claim}

  \newcommand{\NBW}[2][\infty]{\mathrm{NBW}_{#2}^{(#1)}}
  \newcommand{\BadW}[2][\infty]{B_{#2}^{(#1)}}
  \newcommand{\GoodW}[2][\infty]{G_{#2}^{(#1)}}
  \newcommand{\N}[2][\infty]{N_{#2}^{(#1)}}
  \newcommand{\Nc}[2][\infty]{\widetilde{N}_{#2}^{(#1)}}
  \newcommand{\Cy}[2][\infty]{C_{#2}^{(#1)}}
  \newcommand{\CNBW}[2][\infty]{\mathrm{CNBW}_{#2}^{(#1)}}
  \newcommand{\Er}[2][\infty]{E_{#2}^{(#1)}}
  \newcommand{\sCNBW}[2][\infty]{\mathrm{\widetilde{CNBW}}_{#2}^{(#1)}}
  \newcommand{\X}[2][\infty]{X_{#2}^{(#1)}}

\title[Limit theorems]{Functional limit theorems for random regular graphs}

\author{Ioana Dumitriu}
\address{Department of Mathematics\\ University of Washington\\ Seattle, WA 98195}
\email{dumitriu@math.washington.edu}
\author{Tobias Johnson}
\address{Department of Mathematics\\ University of Washington\\ Seattle, WA 98195}
\email{toby@math.washington.edu}
\author{Soumik Pal}
\address{Department of Mathematics\\ University of Washington\\ Seattle, WA 98195}
\email{soumikpal@gmail.com}
\author{Elliot Paquette}
\address{Department of Mathematics\\ University of Washington\\ Seattle, WA 98195}
\email{paquette@math.washington.edu}



\thanks{Soumik's research is partially supported by NSF grant DMS-1007563.
Ioana, Elliot, and Toby acknowledge support from the NSF by means of
CAREER Award DMS-0847661.  Toby's research is also
supported by the ARCS Foundation.}

\date{\today}

\begin{abstract} Consider $d$ uniformly random permutation matrices on
  $n$ labels. Consider the sum of these matrices along with their
  transposes. The total can be interpreted as the adjacency matrix of
  a random regular graph of degree $2d$ on $n$ vertices. We consider
  limit theorems for various combinatorial and analytical properties
  of this graph (or the matrix) as $n$ grows to infinity, either when
  $d$ is kept fixed or grows slowly with $n$. In a suitable weak
  convergence framework, we prove that the (finite but growing in
  length) sequences of the number of short cycles and of cyclically
  non-backtracking walks converge to distributional limits. 
 We estimate the total variation distance from the limit using Stein's
 method. As an application of these results we derive limits of linear
  functionals of the eigenvalues of the adjacency matrix. 
  A key step in this latter derivation is an extension of the Kahn-Szemer\'edi argument for estimating the second largest eigenvalue for all values of $d$ and $n$.  
\end{abstract}

\maketitle

\section{Introduction}
 We consider several asymptotic enumeration and analytic problems for sparse random regular graphs and their adjacency matrices.
  A graph is called regular if every vertex has the same degree; a
  sparse regular graph is typically one for which the degree $d$ is
  either constant or of a far smaller order than the number of
  vertices $n$.  A classical model is the uniform distribution over
  all $d$-regular graphs on $n$ labeled vertices; a thorough survey on
  properties of the uniform model can be found in  \cite{W}. 

Our model of choice is the more recent permutation model: Consider $d$
many iid uniformly random permutations $\{ \pi_1, \ldots, \pi_d \}$ on
$n$ vertices labeled $\{1,2,\ldots,n\}$. A graph can be constructed by
adding one edge between each pair $(i, \pi_j(i))$; thus
every vertex
$i$ has edges to $\pi_j(i)$ and $\pi^{-1}_j(i)$ for every permutation
$\pi_j$, for a total degree of $2d$.
 As the reader will note, this
allows multiples edges and self-loops, with each self-loop contributing
two to the degree of its vertex.  However, one can still ask the
usual enumeration questions about this graph, e.g., the distribution of the number of cycles.   
 
Another way to represent this graph is by its adjacency
matrix, which is an $n\times n$ matrix whose $(i,j)$th entry is the
number of edges between $i$ and $j$, with self-loops counted twice. 
This random matrix can be now
studied in it own right; for example, one can ask about the
distribution of its eigenvalues. Note that---trivially---the top
eigenvalue is $2d$; the distribution of the rest of the eigenvalues is an interesting question. For the uniform model of random regular graphs (or Erd\H{o}s-R\'enyi  graphs) such questions have been studied since the pioneering work \cite{mckay81}. Among the more recent articles, see  \cite{FeigeOfek}, \cite{TVW}, and \cite{DP}. We refer the reader to \cite{DP} for a more exhaustive review of the vast related literature. 

Our results touch on both aspects. We consider two separate scenarios,
either when $d$ is independent of $n$, or when $d$ grows
\textit{slowly} with $n$.  We will assume throughout that $d\ge 2$;
the reason for this is that the $d=1$ case has been dealt with (in a
larger context) by \cite{BAD}. 
 
The paper is divided into three thematically separate but mathematically dependent parts. 
\bigskip

\noindent(i) Section \ref{sec:poisson}: \textbf{Joint asymptotic
  distribution of a growing sequence of short cycles.} It is
  well known in the classical models of random regular graphs
that the number of cycles of length
$k$, where $k$ is \textit{small} (typically logarithmic in $n$), is
approximately Poisson. See
\cite{bollobasbook} or \cite{W} for an account of older results, or 
\cite{mww04} for the best result in this direction.
In Theorem~\ref{thm:multipoiquant}, we prove this fact for the permutation
model, using
Stein's method along with ideas from \cite{LP} to estimate the total
variation distance between a vector of the number of 
cycles of lengths $1$ to $r$ and a corresponding 
vector of independent Poisson random variables.  This theorem
holds for nearly 
the same regime of $r$, $d$, and $n$ as in \cite[Theorem~1]{mww04},
and unlike that theorem gives an explicit error bound on the approximation.  
This bound
is essential to our analysis of eigenvalue statistics
in Section~\ref{sec:linear}.

 The mean number of cycles
is somewhat interesting. When $d$ is fixed, for the uniform model of random 
$2d$-regular graphs, the limiting mean of the number of short cycles of length $k$ is $(2d-1)^k/2k$. For the permutation model, the limiting mean is the slightly different quantity $a(d,k)/2k$, where
\[
a(d,k)=\begin{cases}
(2d-1)^{k} - 1 + 2d,& \text{when $k$ is even},\\
(2d-1)^{k} +1,& \text{when $k$ is odd}.
\end{cases}
\] 
See also \cite[Theorem~4.1]{LMMW}, in which the authors consider a
different model of random regular graph and find that the limiting mean
number of cycles of length $k$ differs slightly from both of these.

Next we consider the number of short non-backtracking walks on the
graph; a non-backtracking walk is a closed walk that never follows an
edge and immediately retraces that same edge backwards. We actually
consider cyclically non-backtracking walks (CNBWs), whose definition
will be given in  Subsection~\ref{sec:nbw}. Non-backtracking walks are
important in both theory and practice as can be seen from the articles
\cite{friedmanalon} and \cite{ABLS}. We consider the entire vector
of cyclically non-backtracking walks of lengths $1$ to $r_n$, where
$r_n$ is the ``boundary length'' of short walks/cycles, and is growing to infinity with $n$. In Theorem \ref{thm:dfixedtight}, we assume that $d$ is independent  of $n$. We prove that the vector of CNBWs, as a random sequence in a weighted $\ell^2$ space, converges weakly to a limiting random sequence whose finite-dimensional distributions are linear sums of independent Poisson random variables.

When $d$ grows slowly with $n$ (slower than any fixed power of $n$,
which is the same regime studied in \cite{DP}), a
corresponding result is proved in Theorem \ref{thm:dgrowstight}. Here,
we \textit{center} the vector of CNBW for each $n$. The resulting
random sequence converges weakly to an infinite sequence of independent, centered normal random variables with unequal ($\sigma_k^2=2k$) variances. 

\bigskip

\noindent(ii) Section \ref{sec:eval2}: \textbf{An estimate of
  $C\sqrt{2d-1}$ for the second largest (in absolute value) eigenvalue
  for  any $(d,n)$.} The spectral gap of the permutation model, for
fixed $d$, has been intensely studied recently in \cite{friedmanalon}
for the resolution of the Alon conjecture. This conjecture states that
the second largest eigenvalue of `most random regular graphs' of
degree $2d$ is less than $2\sqrt{2d-1} + \epsilon$; the assumption is
that $d$ is kept fixed while $n$ grows to infinity. This important conjecture implies that `most' sparse random
regular graphs are nearly Ramanujan (see \cite{LPS}). Friedman's work
builds on earlier work \cite{furedikomlos}, \cite{broder_shamir}, and 
\cite{friedmane2}. Although \cite{friedmanalon} and related works consider the permutation model, for fixed $d$, their results also apply to other models due to various contiguity results; see \cite[Section~4]{W} and \cite{GJKW}.

To develop the precise second eigenvalue control that we require in 
Section~\ref{sec:linear},
we have followed a line of reasoning that originates with Kahn and Szemer\'edi \cite{FKSz}.
This approach has been used recently to great effect by \cite{BFSU}, \cite{FeigeOfek}, and
\cite{LSV}, to name a few.  With this technique we are able to show 
that the second largest eigenvalue is bounded by $40000\sqrt{2d-1}$
with a probability at least $1- Cn^{-1}$ for some universal constant $C$ 
(see Theorem \ref{thm:eigbound}).  We have not attempted to find an optimal constant, and
instead we focus on extricating the $d$ and $n$ dependence in the bound.

Both
\cite{BFSU} and \cite{LSV} provide examples of how the Kahn-Szemer\'edi argument can
be used to control the second eigenvalue when $d$ grows with $n$.  In \cite{BFSU},
the authors work in the configuration model to obtain the $O(\sqrt{d})$ bound for $d = O(\sqrt{n}),$ essentially the largest $d$ for which the configuration model represents the uniform $d$-regular graph well enough to prove eigenvalue concentration.
In \cite{LSV}, the authors study the spectra of random covers.  The permutation model is an
example of such a cover, where the base graph is a single point with $d$ self loops.  Using the
Kahn-Szemer\'edi machinery, they are able to show an $O(\sqrt{d} \log d)$ bound with $d(n) = \text{poly}(n)$ growth.  The adaptations to the original Kahn-Szemer\'edi argument made in \cite{LSV}, especially the usage of Freedman's martingale inequality, are similar to the ones made here.  However, as we do not need to consider the geometry of the base graph, we are able to push this argument to prove a non-asymptotic bound of the correct order.



\bigskip

\noindent(iii) Section \ref{sec:linear}: \textbf{Limiting distribution of linear eigenvalue statistics of the rescaled adjacency matrix.} Our final section is in the spirit of Random Matrix Theory (RMT). Let $A_n$ denote the adjacency matrix of a random regular graph on $n$ vertices. 
By linear statistics of the spectrum we mean random variables of the type $\sum_{i=1}^n f(\lambda_i)$, where $\lambda_1 \ge \ldots \ge \lambda_n$ are the $n$ eigenvalues of the symmetric matrix $(2d-1)^{-1/2}A_n$. We do this rescaling of $A_n$ irrespective of whether $d$ is fixed or growing so as to keep all but the first eigenvalue bounded with high probability.  

The limiting distribution of linear eigenvalue statistics for various
RMT models such as the classical invariant ensembles or the
Wigner/Wishart matrices has been (and continues to be) widely
studied. For the sake of space, we give here only a brief (and
therefore incomplete) list of methods and papers which study
the subject. For a more in-depth review, we refer the reader to \cite{AGZ}. 

The first, and still one of the most widely used methods of approach
is the method of moments, introduced in \cite{wigner55a}, used in
\cite{jonsson82a} and perfected in \cite{sosh_sinai} for Wigner
matrices (it also works for Wishart);  this method
is also used here in conjunction with other tools. Explicit moment calculations
alongside Toeplitz determinants have also been used in determining the
linear statistics of circular ensembles \cite{szego52}, \cite{DE}, \cite{johansson88}.

Other methods include the Stieltjes transform method (also known as the method of resolvents),
which was employed with much success in a series of papers of which we
mention \cite{bai_silverstein04}  and \cite{LP09}; the (quite analytical)
method of potentials, which works on  a different class of
random matrices including the Gaussian Wigner ones
\cite{johansson_clt}; stochastic calculus \cite{cabanal_duvillard01};
and free probability \cite{KMS07}. Finally, a completely different set of techniques were explored in \cite{Cha09}. 

Recently and notably, for a single permutation matrix,  such a study
has been approached in \cite{W00} and completed in \cite{BAD}; our results share
several features with the latter paper.

A noteworthy aspect in all these is that when the function $f$ is smooth enough
(usually analytic),
the variance of the random
variables $\sum_{i=1}^n f(\lambda_i)$ typically remains bounded. This is
attributed to eigenvalue \textit{repulsion}; see \cite[Section~21.2.2]{BAG} for
further discussion. Even more interestingly,
there is no process convergence of the cumulative distribution
function. This can be guessed from the fact that when the function $f$
is rough (e.g., the characteristic function of an interval), the
variance of the linear statistics grows slowly with $n$
(as seen for example in \cite{CL} and \cite{soshnikov00}). 
One major difference our models have with the classical ensembles is
that our matrices are sparse; their sparsity affects the behavior of the limit.       

In Theorems \ref{thm:dfixedlinear} and \ref{thm:dgrowslinear} we prove
limiting distributions of linear eigenvalue statistics. For fixed $d$,
the functions we cover are those that are analytically continuable to
a large enough ellipse containing a compact interval of spectral
support. When $d$ grows we need functions that are slightly more
smooth. Let $(T_k)_{k\in \NN}$ be the Chebyshev polynomials of the
first kind on a certain compact interval; since they constitute a basis for $\ltwo$ functions,
any such function admits a decomposition in a Fourier-like series
expressed in terms of the Chebyshev polynomials. The required smoothness is characterized in terms of how quickly
the truncated series converges in the supremum norm to the actual
function on the given interval.

In Theorem \ref{thm:dfixedlinear}, we consider $d$ to be fixed. The
limiting distribution of the linear eigenvalue statistics is a
non-Gaussian infinitely divisible distribution. This is consistent
with the results in \cite{BAD}. Theorem \ref{thm:dgrowslinear} proves
a Gaussian limit in the case of a slowly growing $d$ after we have
appropriately centered the random variables. This transition is
expected. In \cite{DP} the authors consider the uniform model of
random regular graphs and show that when $d$ is growing slowly, the
spectrum of the adjacency matrix starts resembling that of a real
symmetric Wigner matrix. Similar techniques, coupled with estimates
proved in this paper, could be used to extend such results to the present model.  

The proofs in this section follow easily from the results in parts (i) and (ii) above. As in \cite{DP}, the proofs display interesting combinatorial interpretations of analytic quantities common in RMT.

\section*{Acknowledgment} The authors gratefully acknowledge several
useful discussions with G\'erard Ben Arous, Kim Dang,  Joel Friedman, and Van Vu. In particular, they thank G\'erard for sharing the article \cite{BAD}. Joel and Van have generously pointed us toward the techniques in \cite{FKSz} that we use in Section~\ref{sec:eval2}.  Additionally, we would like thank the referees and the associate editor for many insightful comments.


\section{A weak convergence set-up}\label{sec:weakconvergence} The
following weak convergence set-up will be used to prove the limit
theorems in the later text. Let $\weight:=(\omega_m)_{m\in \NN}$ be a
sequence of positive \textit{weights} that decay to zero at a suitable
rate as $m$ tends to infinity. Let $\ltwo(\weight)$ denote the space
of sequences $(x_m)_{m \in \mathbb{N}}$ that are square-integrable
with respect to $\weight$, i.e., $\sum_{m=1}^{\infty} x_m^2 \omega_m <
\infty$. Our underlying complete separable metric space will be
$X=(\ltwo(\weight), \norm{\cdot})$, where $\norm{\cdot}$ denotes the
usual norm. 

\begin{rmk} Although we have chosen to work with $\ltwo$ for
  simplicity, any $\mathbf{L}^p$ space would have worked as well. 
\end{rmk}

Let us denote the space of probability measures on the Borel
$\sigma$-algebra of $X$ by $\PP(X)$. We will skip mentioning the Borel
$\sigma$-algebra and refer to a member of $\PP(X)$ as a probability
measure on $X$. We equip $\PP(X)$ with the Prokhorov metric for weak
convergence; for the standard results on weak convergence we use below, please consult Chapter 3 in \cite{EK}. Let $\rho$ denote the Prokhorov metric on $\PP(X) \times \PP(X)$ as given in \cite[eqn. (1.1) on page 96]{EK}. 

\begin{lemma}
The metric space $(\PP(X), \rho)$ is a complete separable metric space. 
\end{lemma}

\begin{proof}  
The claim follows from \cite[Thm.~1.7, p.~101]{EK} since $X$ is a complete separable metric space. 
\end{proof}

\bigskip

To prove tightness of subsets of $\PP(X)$ we will use the following class of compact subsets of $\ltwo(\weight)$.

\begin{lemma}[The infinite cube]\label{lemma:hilbertcube} Let $(a_m)_{m\in \NN} \in \ltwo(\weight)$ be such that $a_m \ge 0$ for every $m$. Then the set
\[
\left\{  (b_m)_{m\in \mathbb{N}}\in \ltwo(\weight):\quad 0\le \abs{b_m} \le a_m\quad \text{for all}\quad m \in \NN  \right\}
\]
is compact in $(\ltwo(\weight), \norm{\cdot})$.
\end{lemma}

\begin{proof}
First observe that the cube is compact in the product topology by Tychonoff's theorem. Norm convergence to the limit points now follows by the Dominated Convergence Theorem.   
\end{proof}

We now explore some consequences of relative compactness. 

\begin{lemma}\label{lem:contimap} Suppose $\{X_n\}$ and $X$ are random sequences taking values in $\ltwo(\weight)$ such that $X_n$ converges in law to $X$. Then, for any $b \in \ltwo(\weight)$, the random variables $\iprod{b, X_n}$ converges in law to $\iprod{b,X}$.  
\end{lemma}

\begin{proof} This is a corollary of the usual Continuous Mapping Theorem. 
\end{proof}

Our final lemma shows that \textit{finite-dimensional} distributions characterize a probability measure on the Borel $\sigma$-algebra on $X$.

\begin{lemma}\label{lem:fd-characterize}
Let $x$ be a typical element in $X$. Let $P$ and $Q$ be two probability 
measures on $X$. Suppose for any finite collection of indices 
$(i_1, \ldots, i_k)$, the law of the random vector 
$(x_{i_1}, \ldots, x_{ i_k})$ is the same under both $P$ and $Q$. 
Then $P=Q$ on the entire Borel $\sigma$-algebra.
\end{lemma}

\begin{proof} Our claim will follow once we show that $P$ and $Q$ give
  identical mass to every basic open neighborhood determined by the
  norm; however, the norm function $x \mapsto \norm{x}$ is measurable
  with respect to the $\sigma$-algebra generated by coordinate
  projections. Now, under our assumption, every finite-dimensional
  distribution is identical under $P$ and $Q$; hence the probability measures $P$ and $Q$ are identical on the coordinate $\sigma$-algebra. This proves our claim.
\end{proof}

\section{Some results on Poisson approximation}\label{sec:poisson}
\subsection{Cycles in random regular graphs}
\label{sec:cycles}
 Let $G_n$ be the $2d$-regular graph on $n$ vertices
 sampled from $\mathcal{G}_{n,2d}$, the permutation model
 of random regular graphs.  The graph $G_n$ is  generated 
  from the uniform random permutations $\pi_1,\ldots,\pi_d$ 
  as described in the introduction.
  Assume that the vertices of $G_n$ are labeled by $\{1,\ldots,n\}$, and let $C^{(n)}_k$ denote the number of (simple) cycles
  of length $k$ in $G_n$.  
  
  We start by giving the limiting distribution
  of $C^{(n)}_k$ as $n\to\infty$.  
    Suppose that $w=w_1\cdots w_k$ is a word 
    on the letters $\pi_1,\ldots,\pi_d$ and $\pi_1^{-1},\ldots,\pi_d^{-1}$.
    We call $w$ \emph{cyclically reduced} if $w_1\neq w_k^{-1}$
    and $w_i\neq w_{i+1}^{-1}$ for $1\leq i < k$.
    Let $a(d,k)$ denote the number of cyclically reduced words of length
    $k$ on this alphabet.
  \begin{prop}\label{thm:poi} 
    As $n\to\infty$ while $k$ and $d$ are kept fixed, 
    \begin{align*}
      C^{(n)}_k\toL \Poi\left(\frac{a(d,k)}{2k}\right).
    \end{align*}
  \end{prop}
  We will actually give a stronger version
  of this result in Theorem~\ref{thm:poiquant}, but
  we include this proposition nevertheless because
  it has a more elementary proof, and because in
  proving it we will develop some lemmas that will come in handy later.
  We also note the following exact expression for
  $a(d,k)$,
  \begin{align}
    a(d,2k) &= (2d-1)^{2k} -1+2d,~~\mbox{and}~~
    & a(d,2k+1)&= (2d-1)^{2k+1}+1, \label{eq:adkbounds}
  \end{align}
	whose proof we provide in the Appendix (see Lemma~\ref{adk}).

      Our argument heavily uses the concepts of \cite{LP},
    but we will try to make our proof self-contained.
    Let $\Ww$ be the set of cyclically reduced words of length $k$
    on letters $\pi_1,\ldots,\pi_d$ and $\pi_1^{-1},\ldots,\pi_d^{-1}$.
    For $w\in\Ww$, we define a \emph{closed trail} with
    word $w$ to be an object  of the form
    \begin{align*}
        \xymatrix{
          s_0\ar[r]^{w_1} & s_1\ar[r]^{w_2}&s_2\ar[r]^{w_3}&
          \cdots\ar[r]^{w_k}&s_k=s_0
        }
    \end{align*}
    with $s_i\in\{1,\ldots,n\}$.  In Section~\ref{sec:cycles}, 
    we will consider only
    the case where $s_0,\ldots,s_{k-1}$ are distinct, though
    we will drop this assumption in Section~\ref{sec:nbw}.
    We say that the trail appears in $G_n$ if $w_1(s_0)=s_1$,
    $w_2(s_1)=s_2$, and so on. In other words, we are considering
    $G_n$ as a directed graph with edges labeled by the permutations
    that gave rise to them, and we are asking if it contains
    the trail as a subgraph.
    We note that a trail (with distinct
    vertices) can only appear in $G_n$ if its word is cyclically reduced.
    
    To give an idea of the method we will use, we demonstrate
    how to calculate $\limn\E[C_k^{(n)}]$.  Suppose we have
    a trail with word $w$. Let
    $e_w^i$ be the number of times $\pi_i$ or $\pi_i^{-1}$ appears
    in $w$.  It is straightforward to see that the trail appears in
    $G_n$ with probability
    $\prod_{i=1}^d 1/[n]_{e_w^i}$,
    where $$[x]_j=x(x-1)\cdots(x-j+1)$$ is the falling factorial or
    Pochhammer symbol.
   
 For every word in $\Ww$, there are
    $[n]_k$ trails with that word.  The total number of trails
    of length $k$
    contained in $G_n$ is $2k$ times the number of cycles, so
    \begin{align}
      2k\E[C_k^{(n)}]=\sum_{w\in\Ww}[n]_k\prod_{i=1}^d\frac{1}{[n]_{e_w^i}}.
        \label{eq:ccount}
    \end{align}
    Each summand converges to 1 as $n\to\infty$,
    giving $\E[C_k^{(n)}]\to a(d,k)/2k$, consistent with
    Proposition~\ref{thm:poi}.
    
    To prove Proposition~\ref{thm:poi}, we will need to count
     more complicated objects
    than in the above example, and we will need some machinery
    from \cite{LP}.
    Suppose we have the following list
    of $r$ trails with associated words $w^1,\ldots,w^r$:
    \begin{align}
          &\xymatrix{
          s^1_0\ar[r]^{w^1_1} & s^1_1\ar[r]^{w^1_2}&
          \cdots\ar[r]^{w^1_k}&s^1_k}\nonumber\\
                  &\xymatrix{
          s^2_0\ar[r]^{w^2_1} & s^2_1\ar[r]^{w^2_2}&
          \cdots\ar[r]^{w^2_k}&s^2_k}\label{eq:list}\\
          &\qquad\qquad\qquad\vdots\nonumber\\
          &\xymatrix{
          s^r_0\ar[r]^{w^r_1} & s^r_1\ar[r]^{w^r_2}&
          \cdots\ar[r]^{w^r_k}&s^r_k}\nonumber
    \end{align}
    with $s_i^j\in\{1,\ldots,n\}$.  Though we take the vertices
    $s^j_0,\ldots,s^j_{k-1}$ of each trail to be distinct, vertices
    from different trails may coincide (see Figure~\ref{fig:trailcat}
    for an example).
    
    Suppose we have another list of $r$ trails, $(u_i^j,\ 0\leq i\leq k,
     1\leq j\leq r)$ with the same
    words $w^1,\ldots,w^r$. We say that these two lists are of
    the same category if $s_i^j=s_{i'}^{j'}\iff u_i^j=u_{i'}^{j'}$.
    Roughly speaking, this means that the trails in the two lists
    overlap each other in the same way.  The probability
    that some list of trails appears in $G_n$ depends only on its category.
    
    We can represent each category as a directed, edge-labeled graph 
    depicting the overlap of the trails.  This is more complicated
    to explain than to do, and we encourage the reader
    to simply look at the example in Figure~\ref{fig:trailcat},
    or at Figure~7 in \cite{LP}.  Given the list of trails
    $(s_i^j)$, we define this graph as follows.  First,
    reconsider the variables
    $s_i^j$ simply as abstract labels rather than elements
    of $\{1,\ldots, n\}$, and partition these labels
    by placing any two of them in a block together
    if (considered as integers again) they are equal. The graph 
    has these blocks as its vertices.
    It includes an  edge labeled $\pi_i$
    from one block to another if the trails include
    a step labeled $\pi_i$ or $\pi_i^{-1}$ from any vertex in the
    first block to any vertex in the second; this edge should be
    directed according to whether the step was labeled  $\pi_i$ or
    $\pi_i^{-1}$.
    
    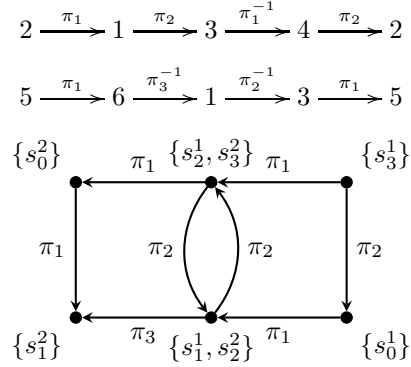
\begin{figure}
    \begin{align*}
        \xymatrix{
          2\ar[r]^{\pi_1} & 1\ar[r]^{\pi_2}&3\ar[r]^{\pi_1^{-1}}&4
          \ar[r]^{\pi_2}&2
        }\\
        \xymatrix{
          5\ar[r]^{\pi_1} & 6\ar[r]^{\pi_3^{-1}}&1\ar[r]^{\pi_2^{-1}}&3
          \ar[r]^{\pi_1}&5
        }
    \end{align*}
    
    \begin{center}
      \begin{tikzpicture}[scale=1.8,vert/.style={circle,fill,inner sep=0,
              minimum size=0.15cm,draw},>=stealth]
        \node[vert, label=below:{$\{s^1_1,s^2_2\}$}] (a1) at (0,0) {};
        \node[vert, label=below right:{$\{s^1_0\}$}] (a2) at (1,0) {};
        \node[vert, label=above:{$\{s^1_2,s^2_3\}$}] (a3) at (0,1) {};
        \node[vert, label=above right:{$\{s^1_3\}$}] (a4) at (1,1) {};
        \node[vert, label=above left:{$\{s^2_0\}$}] (a5) at (-1,1) {};
        \node[vert, label=below left:{$\{s^2_1\}$}] (a6) at (-1,0) {};
        \draw[thick,->] (a2) to node[auto] {$\pi_1$} (a1);
        \draw[thick,->] (a1)
                        to [bend right=38] node[auto,swap] {$\pi_2$} (a3);
        \draw[thick,->] (a4) to node[auto,swap] {$\pi_1$} (a3);
        \draw[thick,->] (a4) to node[auto] {$\pi_2$} (a2);
        \draw[thick,->] (a3) to node[auto,swap] {$\pi_1$} (a5);
        \draw[thick,->] (a5) to node[auto,swap] {$\pi_1$} (a6);
        \draw[thick,->] (a1) to node[auto] {$\pi_3$} (a6);
        \draw[thick,->] (a3) to [bend right=38] node[auto,swap] {$\pi_2$} (a1);          
      \end{tikzpicture}
    \end{center}
    \caption{A list of two trails, and the graph associated with its
    category.  Since $s_2^1=s_3^2=3$, the vertices
    $s_2^1$ and $s_3^2$ are blocked together in the graph,
    and since $s_1^1=s_2^2=1$, the vertices $s_1^1$
    and $s_2^2$ are blocked together.}\label{fig:trailcat}
    \end{figure}
    Suppose that $\Gamma$ is the graph of 
    a category of a list
    of trails, and define $X_\Gamma^{(n)}$ to be the number
    of tuples of trails of category $\Gamma$ found in $G_n$.
    If $\Gamma$ is the graph of a category of a list of a \emph{single}
    trail with word $w\in\Ww$, we write $X_w^{(n)}$ for $X_\Gamma^{(n)}$.
    Note that such graphs have a simple form demonstrated in
    Figure~\ref{fig:singlecycle}.
    
    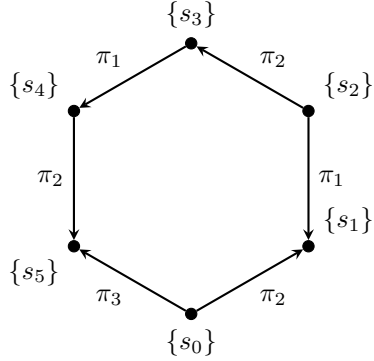
\begin{figure}
      \begin{center}
        \begin{tikzpicture}[scale=1.8,vert/.style={circle,fill,inner sep=0,
              minimum size=0.15cm,draw},>=stealth]
          \node[vert,label=270:{$\{s_0\}$}] (s0) at (270:1) {};
            \node[vert,label=33:{$\{s_1\}$}] (s1) at (330:1) {};
            \node[vert,label=30:{$\{s_2\}$}] (s2) at (30:1) {};
            \node[vert,label=90:{$\{s_3\}$}] (s3) at (90:1) {};
            \node[vert,label=150:{$\{s_4\}$}] (s4) at (150:1) {};
            \node[vert,label=210:{$\{s_5\}$}] (s5) at (210:1) {};
            \draw[thick,->] (s0) to node[auto,swap] {$\pi_2$} (s1);
            \draw[thick,<-] (s1) to node[auto,swap] {$\pi_1$} (s2);
            \draw[thick,->] (s2) to node[auto,swap] {$\pi_2$} (s3);
            \draw[thick,->] (s3) to node[auto,swap] {$\pi_1$} (s4);
            \draw[thick,->] (s4) to node[auto,swap] {$\pi_2$} (s5);
            \draw[thick,<-] (s5) to node[auto,swap] {$\pi_3$} (s0);
        \end{tikzpicture}
      \end{center}
      \caption{The graph $\Gamma$ associated with
      a single trail with word $\pi_2\pi_1^{-1}\pi_2\pi_1
         \pi_2\pi_3^{-1}$.}
      \label{fig:singlecycle}
    \end{figure}
    
    \begin{lemma}\label{lem:10}
      \begin{align*}
         \limn\E[X_\Gamma^{(n)}]=\begin{cases}
           1&\text{if $\Gamma$ has the same number of vertices
             as edges}\\
           0&\text{otherwise}
         \end{cases}                        
       \end{align*}
    \end{lemma}
    To demonstrate the connection to the calculation
    we performed in \eqref{eq:ccount}, observe that
    \begin{align*}
      2kC_k^{(n)}=\sum_{w\in\Ww}X_w^{(n)},
    \end{align*}
    and by our lemma the expectation of
    this converges to $a(d,k)$ as $n\to\infty$.
    \begin{proof}[Proof of Lemma~\ref{lem:10}]
      This is essentially the same calculation as in
      \eqref{eq:ccount}.  Let $e$ and $v$ be the number
      of edges and vertices, respectively, of the
      graph $\Gamma$.  Let $e_i$ be the number of edges
      in $\Gamma$ labeled by $\pi_i$.
      
      There are $[n]_v$ different trails of category $\Gamma$,
      corresponding to the number of ways to assign
      vertices $\{1,\ldots, n\}$ to the vertices of $\Gamma$.
      Since each of these trails appears in $G_n$
      with probability $\prod_{i=1}^d1/[n]_{e_i}$, 
      \begin{align}
        \E[X_\Gamma^{(n)}] = [n]_v\prod_{i=1}^d\frac{1}{[n]_{e_i}}.
          \label{eq:10}
      \end{align}
      As $n\to\infty$, this converges to 0 if $e>v$ and to 1 if
      $e=v$.  If $\Gamma$ is the graph of a category of
      a list of trails, then every vertex has degree
      at least 2, so it never happens that $e<v$, which completes
      the lemma.  We note for later use 
      that this remains true even when we drop
      the requirement that all vertices of a trail be distinct,
      so long as the word of each trail is cyclically reduced.
    \end{proof}
      
  \begin{proof}[Proof of Proposition~\ref{thm:poi}]
    We will use the moment method.  Fix a positive integer $r$.
    The main idea of the proof is
    interpret $\big(\Cy[n]{k}\big)^r$ as the number of
    $r$-tuples of cycles of length $k$ in $G_n$.  As there are $2k$ closed
    trails for every cycle of length $k$, we can also think
    of it as $(2k)^{-r}$ times the number of $r$-tuples of closed trails
    of length $k$ in $G_n$.
    
    Let $\Gg$ be the set of graphs of categories of lists of $r$ trails
    of length $k$.
    The above interpretation implies that
    \begin{align}
      \big(\Cy[n]{k}\big)^r = \frac{1}{(2k)^r}\sum_{\Gamma\in\Gg}\X[n]{\Gamma}.
      \label{eq:tupleinterp}
    \end{align}
    By Lemma~\ref{lem:10}, we can compute $\limn \E\big(\Cy[n]{k}\big)^r$
    by counting the number of graphs in $\Gamma$ with the same number
    of edges as vertices.  Let $\Gg'\subset\Gg$ be the set of
    such graphs.
    
    Let $\Gamma\in\Gg'$, and consider some list of $r$
    trails of category $\Gamma$.  Since $\Gamma$ has as
    many edges as vertices, it consists of disjoint cycles.
    This implies that for any two trails in the list, either the
    trails are wholly identified in $\Gamma$, or they are are
    disjoint.  These identifications of the $r$ different trails
    give a partition of $r$ objects.
    
    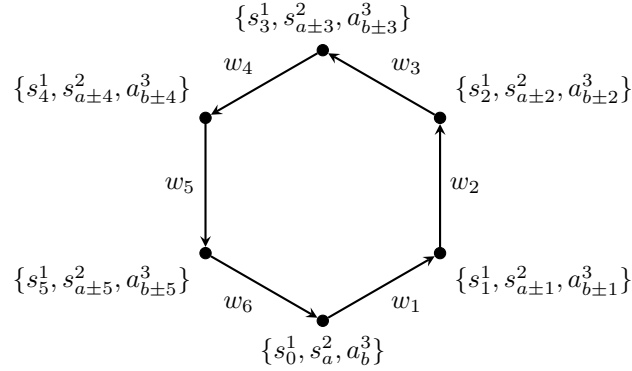
\begin{figure}
      \begin{center}
        \begin{tikzpicture}[scale=1.8,vert/.style={circle,fill,inner sep=0,
              minimum size=0.15cm,draw},>=stealth]
          \node[vert,label=270:{$\{s^1_0, s^2_a, a^3_b\}$}] (s0) at (270:1) {};
            \node[vert,label=-33:{$\{s^1_1, s^2_{a\pm 1}, a^3_{b\pm 1}\}$}] 
                (s1)  at (330:1) {};
            \node[vert,label=30:{$\{s^1_2, s^2_{a\pm 2}, a^3_{b\pm 2}\}$}] (s2) at (30:1) {};
            \node[vert,label=90:{$\{s^1_3,s^2_{a\pm 3}, a^3_{b\pm 3}\}$}] (s3) at (90:1) {};
            \node[vert,label=150:{$\{s^1_4,s^2_{a\pm 4}, a^3_{b\pm 4}\}$}] (s4) at 
                    (150:1) {};
            \node[vert,label=210:{$\{s^1_5,s^2_{a\pm 5}, a^3_{b\pm 5}\}$}] (s5) at 
            (210:1) {};
            \draw[thick,->] (s0) to node[auto,swap] {$w_1$} (s1);
            \draw[thick,->] (s1) to node[auto,swap] {$w_2$} (s2);
            \draw[thick,->] (s2) to node[auto,swap] {$w_3$} (s3);
            \draw[thick,->] (s3) to node[auto,swap] {$w_4$} (s4);
            \draw[thick,->] (s4) to node[auto,swap] {$w_5$} (s5);
            \draw[thick,->] (s5) to node[auto,swap] {$w_6$} (s0);
        \end{tikzpicture}
      \end{center}      
      \caption{A graph formed from three trails of length~$6$,
        all identified with each other.  There are $a(d,6)$
        choices for the edge-labels $w$.
        There are six choices for which element $s^2_a$ will
        be identified with $s^1_0$, and two choices for how to orient
        the trail $s^2$ when identifying it with $s^1$.
        There are also six choices for which element $s^3_b$
        will be identified with $s^1_0$, along with another two 
        choices for its orientation.  All together, there
        are $a(d,6)(2\cdot 6)^2$ elements of $\Gg'$ corresponding
        to the partition of three elements into one part.}
        \label{fig:gcounting}
    \end{figure}
      
    Given some partition of the $r$ objects into $m$ parts, we will
    count the graphs in $\Gg'$ whose trails are identified
    according to the partition (see Figure~\ref{fig:gcounting}
    for an example).  Consider some part
    consisting of $p$ trails.  The trails form
    a cycle in $\Gamma$; we need to count the number of different
    ways to label the edges and vertices.
    There are $a(d,k)$ different ways to label the edges.
    Each trail in the part 
    can have its vertices identified with those of the first trail
    in $2k$ different ways, for a total of $(2k)^{p-1}$ choices.
    Thus the number of choices for this part is $a(d,k)(2k)^{p-1}$.
    Doing this for every part in the given partition, we have
    a total of $a(d,k)^m(2k)^{r-m}$.
    Recalling that the number of partitions of $r$ objects into
    $m$ parts is given by the Stirling number of the second kind
    $S(r,m)$,
    \begin{align*}
      |\Gg'| &= \sum_{m=1}^r S(r,m)a(d,k)^m(2k)^{r-m}.
    \end{align*}
    By \eqref{eq:tupleinterp} and Lemma~\ref{lem:10},
    \begin{align*}
      \limn \E\big(\Cy[n]{k}\big)^r&=\sum_{m=1}^r S(r,m)
        \left(\frac{a(d,k)}{2k}\right)^m.
    \end{align*}
    It is well known that this is the $r$th moment
    of the $\Poi(a(d,k)/2k)$ distribution (see for example
    \cite{Pit}), and that this distribution is determined
    by its moments, thus proving the theorem.
  \end{proof}


  This theorem tells us the limiting distribution of $C_k^{(n)}$ as $n\to\infty$,
  with $d$ and $k$ fixed, but tells us nothing if $d$ and $k$ grow with
  $n$.  The following theorem addresses this, and gives us a quantitative
  bound on the rate of convergence.  We will assume throughout
  that $d\geq 2$; we use this assumption only to simplify some
  of our asymptotic quantities, but as far better results
  for the $d=1$ case are already known (see \cite{AT}),
  we see no reason to complicate things.
  For clarity, we state this and future results
  with an explicit constant rather than big-O notation,
  but it is the order, not the constant, that interests us.
  Recall that the total variation distance between two probability measures
    is the largest possible difference between the probabilities that
    they assign to the same event.
  \begin{thm}\label{thm:poiquant}
      There is a constant $C_0$ such that for any $n$, $k$, and $d\geq 2$, the
      total variation distance between the law
      of $C_k^{(n)}$ and $\Poi(a(d,k)/2k)$
      is bounded by $C_0k(2d-1)^k/n$.
    \end{thm}  
  \begin{proof}
  We will prove this using Stein's method;
good introductions to Stein's method for the Poisson
  distribution can be found in
  \cite{CDM},
  \cite{BC}, and especially \cite{BHJ}, which focuses on the
  the technique of size-biased coupling
  that we will employ.   
We give here the basic set-up. 
  Let $\ZZ_+$ denote the nonnegative integers.  
  For any $A\subset\ZZ_+$, let $g=g_{\lambda,A}$
  be the function on $\ZZ_+$ satisfying
  \begin{align*}
    \lambda g(j+1)-jg(j)=\one_{j\in A}-\Poi(\lambda)\{A\}
  \end{align*}
  with $g(0)=0$, where $\Poi(\lambda)\{A\}$ denotes the measure
  of $A$ under the $\Poi(\lambda)$ distribution.  
  This function $g$ is the called the solution to the
  Stein equation.   For any nonnegative
  integer-valued random variable $X$,
  \begin{align}
    \P[X\in A]-\Poi(\lambda)\{A\}=\E[\lambda g(X+1)-Xg(X)].\label{eq:tvdist}
  \end{align}
  Bounding the right hand side of this equation over
  all choices of $g$ thus bounds
  the total variation distance between the law of $X$ and the $\Poi(\lambda)$
  distribution.
  The following estimates on $g$ are standard (see \cite[Lemma~1.1.1]{BHJ},
  for example):
  \begin{align}
    \norm{g}_{\infty}\leq\min(1,\lambda^{-1/2}),\qquad \Delta g\leq\min
      (1,\lambda^{-1}),  \label{eq:steinbounds}
  \end{align}
  where $\Delta g=\sup_{j}|g(j+1)-g(j)|$.  
  
  Let $\Cc$ be the set of closed trails of length $k$ on $n$ vertices,
  with two trails identified if one is a cyclic or inverted cyclic
  shift of one another.
  Elements of $\Cc$ are essentially
  cycles in the complete graph on $n$ vertices, with edges labeled
  by $\pi_1,\ldots,\pi_d$ and $\pi_1^{-1},\ldots,\pi_d^{-1}$.
  We note that $|\Cc|=[n]_ka(d,k)/2k$.
  
  For $t\in\Cc$, let $F_t=\mathbf{1}_{(\text{$t$ occurs in $G_n$})}$.
  Let $\lambda=a(d,k)/2k$.  We abbreviate
  $C_k^{(n)}$ to $C$, and we note that $C=\sum_{t\in\Cc}F_t$.
  We can evaluate the right hand side of \eqref{eq:tvdist}
  as
  \begin{align*}
    \E[\lambda g(C+1)-C g(C)]
      &=\sum_{s\in\Cc}\left(\frac{1}{[n]_k}\E[g(C+1)]-\E[F_sg(C)]\right)
  \end{align*}
  Let $p_t=\E[F_t]$.
  We note that $F_sg(C)=F_sg\big(\sum_{t\neq s}F_t+1\big)$, and
  that
  \begin{align*}
    \E\Big[F_sg\Big(\sum_{t\neq s}F_t+1\Big)\Big]
    &=p_s\E\Big[g\Big(\sum_{t\neq s}F_t+1\Big)\ \Big|\ F_s=1\Big].
  \end{align*}
  In Lemma~\ref{prop:coupling},
  we will construct for each $s\in\Cc$ a random variable
  $Y_s$ on the same probability space as $C$ that has the distribution
  of $\sum_{t\neq s}F_t$ conditioned on $F_s=1$.  Then we evaluate
  \begin{align*}
    |\E[\lambda g(C+1)-C g(C)]|
    &=\left|
    \sum_{s\in\Cc}\left(\frac{1}{[n]_k}\E[g(C+1)]-p_s\E[g(Y_s+1)]\right)
      \right|\\
    &\leq\sum_{s\in\Cc}\frac{1}{[n]_k}\E\big|g(C+1)-g(Y_s+1)\big|
    +\sum_{s\in\Cc}\left|\frac{1}{[n]_k}-p_s\right|\E\big|g(Y_s+1)\big|.
  \end{align*}
  We bound these terms as follows:
  \begin{align*}
    |g(C+1)-g(Y_s+1)|&\leq\Delta g |C-Y_s|\\\intertext{and}
    \left|\frac{1}{[n]_k}-p_s\right|&\leq\left|\frac{1}{[n]_k}-\frac{1}{n^k}
    \right|\leq\frac{k^2}{2n[n]_k}.
  \end{align*}
  This last bound makes use of  the inequality $[n]_k\geq
      n^k(1-k^2/2n)$.  Applying these bounds gives
  \begin{align}
    |\E[\lambda g(C+1)-C g(C)]|
    &\leq\sum_{s\in\Cc}\frac{\Delta g}{[n]_k}\E|C-Y_s|
    +|\Cc|\frac{k^2}{2n[n]_k}\norm{g}_{\infty}\nonumber\\  
    &\leq \frac{\Delta g}{[n]_k}\sum_{s\in\Cc}\E|C-Y_s|
    +O\left(\frac{k^{3/2}(2d-1)^{k/2}}{n}\right).\label{eq:finalbound}
  \end{align}
  To get a good bound on this, we just need to demonstrate how to
  construct $Y_s$ so that $\E|C-Y_s|$
  is small.
  We sketch our method as follows: Fix $s\in\Cc$,
  and let $G_n'$ be a random graph
  on $n$ vertices distributed as $G_n$ conditioned to contain
  the cycle $s$.  We will couple $G_n'$ with $G_n$
  in a natural way, and then prove in Lemma~\ref{lem:coupling}
  that $G_n$ and $G_n'$ differ only slightly.
  We then define $Y_s$ in terms of $G_n'$, and we
  establish
  in Lemma~\ref{prop:coupling} that
  $\E|C-Y_s|$ is small.  Finally,
  we finish the proof of Theorem~\ref{thm:poiquant}
  by using these results to bound the right side
  of \eqref{eq:finalbound}.
  
  We start by constructing $G_n'$.  
  Fix some $s\in\Cc$.
  The basic 
  idea is to modify the permutations
    $\pi_1,\ldots,\pi_d$ to 
    get random permutations $\pi_1',\ldots,\pi_d'$,
    which we will then use to create
    a $2d$-regular graph $G_n'$ in
    the usual way.
    Before we give our construction
    of $\pi_1',\ldots,\pi_d'$, we
    consider what distributions
    they should have.
    Suppose for example that
    $d=3$ and $s$ is
    \begin{align*}
        \xymatrix{
          1\ar[r]^{\pi_3} & 2\ar[r]^{\pi_1^{-1}}&3\ar[r]^{\pi_3}&4\ar[r]^{\pi_1}&1.
        }
    \end{align*}
    To force $G_n'$ to contain $s$,
    $\pi_1'$ should be a uniform random permutation conditioned
    to make $\pi_1'(4)=1$ and $\pi_1'(3)=2$, $\pi_2'$ a uniform
    random permutation with no conditioning, 
    and $\pi_3'$ a uniform random 
    permutation conditioned to make $\pi_3'(1)=2$ and $\pi_3'(3)=4$.

    We now describe the construction of $\pi_1',\ldots,\pi_d'$.
    Suppose $s$ has the form
    \begin{align}\label{eq:sform}
        \xymatrix{
          s_0\ar[r]^{w_1} & s_1\ar[r]^{w_2}&s_2\ar[r]^{w_3}&
          \cdots\ar[r]^{w_k}&s_k=s_0.
        }
    \end{align}
    (The element $s$ is actually an equivalence class of the $2k$
    different cyclic and inverted cyclic shifts of the above trail,
    but we will continue
    to represent it as above.)
    Let $1\leq l\leq d$, and
    suppose that the edge-labels $\pi_l$ and $\pi_l^{-1}$
    appear $M$ times in the cycle $s$, and let
    $(a_m,b_m)$ for $1\leq m\leq M$ be these edges.
    If $(a_m,b_m)$ is labeled $\pi_l$, then $a_m$ is the
    tail and $b_m$ the head of the edge; if it is labeled $\pi_l^{-1}$,
    then $a_m$ is the head and $b_m$ the tail.
    We must construct $\pi_l'$ to have the uniform
    distribution conditioned on $\pi_l'(a_m)=b_m$
      for  $(a_m,b_m),\ 1\leq m\leq M$.
      
      We define a sequence of random transpositions by the following
      algorithm: Let $\tau_1$ swap $\pi_l(a_1)$ and $b_1$.
      Let $\tau_2$ swap $\tau_1\pi_l(a_2)$ and $b_2$, and so on.
      We then define $\pi_l'=\tau_M\cdots\tau_1\pi_l$.
      This permutation satisfies $\pi_l'(a_m)=b_m$ for $1\leq m\leq M$,
      and it is distributed uniformly, subject to the
      given constraints, which is easily 
      proven by induction on each swap.
      This completes our construction of $\pi_1',\ldots,\pi_d'$.
      
      We now define $G_n'$ to be the random graph
      on $n$ vertices with edges $(i,\pi_j'(i))$
      for every $1\leq i\leq n$ and $1\leq j\leq d$.
      It is evident that $G_n'$ is defined
      on the same probability space as $G$
      and is distributed as $G_n$ conditioned on
      containing $s$.
    The key fact
    is that $G_n'$ is nearly
    identical to $G_n$:
    \begin{lemma}\label{lem:coupling}
      Suppose there is an edge $\xymatrix{i\ar[r]^{\pi_l}&j}$
      contained in $G_n$ but not in $G_n'$.  Then the trail $s$
      contains either an edge of the form $\xymatrix{i\ar[r]^{\pi_l}&v}$
      or of the form $\xymatrix{v\ar[r]^{\pi_l}&j}$.
    \end{lemma}
    \begin{proof}
      Suppose
      $\pi_l(i)=j$, but $\pi_l'(i)\neq j$.  Then $j$ must
      have been swapped 
      when making $\pi'_l$, which can happen only if
      $\pi_l(a_m)=j$ or $b_m=j$ for some $m$.  In the first case,
      $a_m=i$ and
      $s$ contains the edge $\xymatrix{i\ar[r]^{\pi_l}&b_m}$ with
      $b_m\neq j$,
      and in the second $s$ contains the edge
      $\xymatrix{a_m\ar[r]^{\pi_l}&j}$ with $a_m\neq i$.
    \end{proof}
    If $s$ contains an edge of the form $\xymatrix{i\ar[r]^{\pi_l}&v}$
    or of the form $\xymatrix{v\ar[r]^{\pi_l}&j}$, then
    $G_n'$ cannot possibly contain $\xymatrix{i\ar[r]^{\pi_i}&j}$ while
    still containing $s$.  The preceding lemma then says that we
    have coupled $G_n$ and $G_n'$ as best we can, in the
    following sense: 
    $G_n'$ keeps as many edges of $G_n$ that it can, given that it 
    contains $s$.
    
    For $t\in\Cc$, let
      $F'_t=\one_{(\text{$G_n'$ contains $t$})}$.
      Define $Y_s$ by $Y_s=\sum_{t\neq s}F_t'$.
    Since $G_n'$ is distributed as $G_n$ conditioned
    to contain $s$, the random variable $Y_s$
    is distributed as $\sum_{t\neq s}F_t$ conditioned
    on $F_s=1$.
    We now proceed to bound $\E|C-Y_s|$, adding
    in the minor technical condition
    that $k<n^{1/6}$.
  
  \begin{lemma}\label{prop:coupling}
    There exists an absolute constant $C_1$ with the
    following property.
    For any $s\in\Cc$ and $Y_s$ defined above, and
    for all $n$, $k$, and $d\geq 2$ satisfying $k<n^{1/6}$,
    \begin{align}
      \E|C-Y_s|\leq  \frac{C_1k(2d-1)^k}{n}\label{eq:couplebound},
    \end{align}
  \end{lemma}
    
    \begin{proof}
      We start by partitioning the cycles of $\Cc$ according to how
      many edges they share with $s$.  Define
      $\Cc_{-1}$ as all elements in $\Cc$ that contain an edge
          $\xymatrix{s_i\ar[r]^{w_i}&v}$ with $v\neq s_{i+1}$
          or an edge $\xymatrix{v\ar[r]^{w_{i+1}}&s_{i+1}}$ 
with $v\neq s_i$.
      For $0\leq j<k$, define $\Cc_j$ as all elements of $\Cc\setminus\Cc_{-1}$
      that share exactly $j$ edges with $s$.  
      
      The sets $\Cc_{-1},\ldots,
      \Cc_{k-1}$ include every element of $\Cc$ except for $s$.
      Loosely, this classifies elements of $\Cc$ according
      their likelihood of appearing in $G_n'$ compared to in $G_n$: trails
      in $\Cc_{-1}$ never appear in $G_n'$; trails in $\Cc_0$ appear
      in $G_n'$ with nearly the same probability as in $G_n$; and 
      the trails in $\Cc_i$ appear in $G_n'$ considerably more often than
      in $G_n$.
      
      This classification of elements of $\Cc$ works nicely with our coupling.
      Suppose $t\in\Cc_i$ for $i\geq 0$.  Lemma~\ref{lem:coupling}
      shows that if $t$ appears in $G_n$, it must also appear in $G_n'$.
      That is,
      $F_t'\geq F_t$ for all $t\in\Cc_i$ for $i\geq 0$. On the
      other hand, $F_t'=0$
      for all $t\in\Cc_{-1}$.  Using this,
      \begin{align}
        \E|C-Y_s| &= \E\left|F_s+\sum_{t\in\Cc_{-1}}(F_t-F'_t)
          +\sum_{t\in\Cc_0}(F_t-F'_t)+\sum_{i=1}^{k-1}\sum_{t\in\Cc_i}
          (F_t-F'_t)\right|\nonumber\\
          &\leq p_s+\E\left|\sum_{t\in\Cc_{-1}}(F_t-F'_t)\right|
          +\E\left|\sum_{t\in\Cc_0}(F_t-F'_t)\right|
          +\E\left|\sum_{i=1}^{k-1}\sum_{t\in\Cc_i}(F_t-F'_t)\right|\nonumber\\
        &= p_s+\sum_{t\in\Cc_{-1}}\E[F_t]+\sum_{t\in\Cc_0}
          \E[F_t'-F_t]+\sum_{i=1}^{k-1}\sum_{t\in\Cc_i}\E[F'_t-F_t]\nonumber\\
          &\leq p_s+\sum_{t\in\Cc_{-1}}p_t+\sum_{t\in\Cc_0}
          (p_t'-p_t)+\sum_{i=1}^{k-1}\sum_{t\in\Cc_i}p'_t,\label{eq:mainbound}
      \end{align}
      with $p'_t=\E[F'_t]$.
      
      The rest of the proof is an analysis of $|\Cc_i|$ and of $p_t'$.
      We start by considering the first sum.  For any edge
      $\xymatrix{s_i\ar[r]^{w_{i+1}}&v}$ with $v\neq s_{i+1}$
          or  $\xymatrix{v\ar[r]^{w_{i+1}}&s_{i+1}}$ with 
          $v\neq s_i$, there are no more than
          $[n-2]_{k-2}(2d-1)^{k-1}$ trails containing that edge (identifying
          cyclic and inverted cyclic shifts).
      This gives the bound
      \begin{align*}
        |\Cc_{-1}|\leq 2k(n-2)[n-2]_{k-2}(2d-1)^{k-1}.
      \end{align*}
      Applying $p_t\leq 1/[n]_k$, 
      \begin{align*}
        \sum_{t\in\Cc_{-1}}p_t =O\left(\frac{k(2d-1)^{k-1}}{n}\right).
       \end{align*}
       For the next sum, we note that with $e_t^i$ denoting
       the number of times $\pi_i$ and $\pi_i^{-1}$ appear in
       the word of $t$, for
       for any $t\in\Cc_0$,
       \begin{align*}
         p_t=\prod_{i=1}^d\frac{1}{[n]_{e_t^i}},\qquad\qquad
         p'_t=\prod_{i=1}^d\frac{1}{[n-e_s^i]_{e_t^i}}.
       \end{align*}
       Thus we have $p'_t\leq 1/[n-k]_k$
       and $p_t\geq 1/n^k$.  Using the bound 
       $|\Cc_0|\leq |\Cc|=a(d,k)[n]_k/2k$,
       we have
       \begin{align*}
         \sum_{t\in\Cc_0}
          (p_t'-p_t)&\leq\frac{a(d,k)[n]_k}{2k}\left(\frac{1}{[n-k]_k}-
          \frac{1}{n^k}\right)\\
          &=\frac{a(d,k)}{2k}\left(\left(\frac{n}{n-k}\right)^k
          \left(1+O\left(\frac{k^2}{n}\right)\right)-
          \left(1+O\left(\frac{k^2}{n}\right)\right)
          \right)\\
          &=\frac{a(d,k)}{2k}\left(\left(1+\frac{k}{n-k}\right)^k
          \left(1+O\left(\frac{k^2}{n}\right)\right)-
          \left(1+O\left(\frac{k^2}{n}\right)\right)
          \right)\\
          &=O\left(\frac{k(2d-1)^k}{n}\right).
       \end{align*}
       
       The last and most involved calculation
       is to bound $|\Cc_i|$.  Fix some choice of $i$ edges
       of $s$.  We start by counting the number of cycles in $\Cc_i$
       that share exactly these edges with $s$.
       We illustrate this process in Figure~\ref{fig:graphassembly}.
       Call the graph consisting of these edges
       $H$, and suppose that $H$ has $p$ components.
       Since it is a forest, $H$ has $i+p$ vertices.
       
       Let $A_1,\ldots, A_p$ be the components of $H$.  We can
       assemble any  $t\in\Cc_i$ that overlaps with $s$ in $H$
       by stringing together these components in some order, with
       other edges in between.  Each component can appear in
       $t$ in one of two orientations.  Since we consider $t$ only up
       to cyclic shift and inverted cyclic shift, we can
       assume without
       loss of generality that
       $t$ begins with component $A_1$ with a fixed orientation.
       This leaves $(p-1)!2^{p-1}$ choices for the order
       and orientation of $A_2,\ldots,A_p$ in $t$.
       
       Imagine now the components laid out in a line,
       with gaps between them, and count the number
       of ways to fill the gaps.
       Each of the
       $p$ gaps
       must contain at least one edge, and the total number
       of edges in all the gaps is $k-i$.
       Thus the total number of possible gap sizes
       is the number of compositions of $k-i$ into $p$ parts,
       or $\binom{k-i-1}{p-1}$.
       
       Now that we have chosen the number of edges
       to appear in each gap, we choose the edges themselves.
       We can do this by giving
       an ordered list $k-p-i$ vertices to go in the gaps,
       along with a label and an orientation
       for each of the $k-i$ edges this gives.
       There are $[n-p-i]_{k-p-i}$ ways to choose the vertices.
       We can give each new edge any orientation and label subject
       to the constraint that the word of $t$ must be reduced.
       This means we have at most $2d-1$ choices for the orientation
       and label of each new edge, for a total of at most $(2d-1)^{k-i}$.
       \begin{figure}
         \begin{center}
        \begin{tikzpicture}[scale=2.1,vert/.style={circle,fill,inner sep=0,
              minimum size=0.15cm,draw}, H/.style={dashed},>=stealth]
          \begin{scope}
          \node[vert,label=90:{$1$}] (s1) at (270:1cm) {};
            \node[vert,label=57:{$2$}] (s2) at (237:1) {};
            \node[vert,label=18:{$3$},H] (s3) at (205:1) {};
            \node[vert,label=358:{$4$},H] (s4) at (171:1) {};
            \node[vert,label=319:{$5$},H] (s5) at (139:1) {};
            \node[vert,label=278:{$6$}] (s6) at (106:1) {};
            \node[vert,label=260:{$7$},H] (s7) at (74:1) {};
            \node[vert,label=221:{$8$},H] (s8) at (41:1) {};
            \node[vert,label=182:{$9$},H] (s9) at (8:1) {};
            \node[vert,label=158:{$10$},H] (s10) at (335:1) {};
            \node[vert,label=118:{$11$}] (s11) at (302:1) {};
            \draw[thick,->] (s1) to node[auto] {$\pi_1$} (s2);
            \draw[thick,->] (s2) to node[auto] {$\pi_1$} (s3);
            \draw[thick,<-,H] (s3) to node[auto,black] {$\pi_2$} (s4);
            \draw[thick,->,H] (s4) to node[auto,black] {$\pi_3$} (s5);
            \draw[thick,->] (s5) to node[auto] {$\pi_1$} (s6);
            \draw[thick,<-] (s6) to node[auto,black] {$\pi_2$} (s7);
            \draw[thick,<-,H] (s7) to node[auto,black] {$\pi_1$} (s8);
            \draw[thick,<-] (s8) to node[auto,black] {$\pi_2$} (s9);
            \draw[thick,->,H] (s9) to node[auto,black] {$\pi_1$} (s10);
            \draw[thick,->] (s10) to node[auto,black] {$\pi_3$} (s11);
            \draw[thick,->] (s11) to node[auto,black] {$\pi_3$} (s1);
            \draw[yshift=-0.7in,text width=6.5cm] node
            {
              The cycle $s$, with $H$ dashed.  The subgraph
              $H$ has components $A_1,\ldots,A_p$.  In this example,
              $p=3$, $k=11$, and $i=4$.
            };
          \end{scope}
          \begin{scope}[xshift=3.7cm]
            \path (-1.476,0) node[vert,label=below:{$3$}] (s3) {}
                  -- ++(0.563,0) node[vert,label=below:$4$] (s4) {}
                  -- ++(0.563,0) node[vert,label=below:$5$] (s5) {}
                  -- ++(0.35,0) node[vert,label=below:$10$] (s10) {}
                  -- ++(0.563,0) node[vert,label=below:$9$] (s9){}
                  -- ++(0.35,0) node[vert,label=below:$7$] (s7) {}
                  -- ++(0.563,0) node[vert,label=below:$8$] (s8){};
            \draw[thick,<-] (s3) to node[auto] {$\pi_2$} (s4);
            \draw[thick,->] (s4) to node[auto] {$\pi_3$}(s5);
            \draw[thick,->] (s9) to node[auto,swap] {$\pi_1$}(s10);
            \draw[thick,<-] (s7) to node[auto] {$\pi_1$}(s8);
            \draw[yshift=-0.7in,text width=6.5cm] (0,0) node
            {
              \textbf{Step 1.} We lay out the components $A_1,\ldots,A_p$.
              We can order and orient $A_2,\ldots,A_p$ however we would like,
              for a total of $(p-1)!2^{p-1}$ choices.
              Here, we have ordered the components $A_1, A_3, A_2$, 
              and we have reversed the orientation of $A_3$.
            };
          \end{scope}
          \begin{scope}[yshift=-4cm,new/.style={}]
            \path (-1.476,0) node[vert,label=below:{$3$}] (s3) {}
                  -- ++(0.563,0) node[vert,label=below:$4$] (s4) {}
                  -- ++(0.563,0) node[vert,label=below:$5$] (s5) {}
                  -- ++(0.35,0) node[vert,label=below:$10$] (s10) {}
                  -- ++(0.563,0) node[vert,label=below:$9$] (s9){}
                  -- +(-0.1065,0.553) node[vert,new] (a1){}
                  -- +(0.4565,0.553) node[vert,new] (a2) {}
                  -- ++(0.35,0) node[vert,label=below:$7$] (s7) {}
                  -- ++(0.563,0) node[vert,label=below:$8$] (s8){}
                  -- (0.492,-0.5) node[vert,new] (a3) {}
                  -- (-0.492,-0.5) node[vert,new] (a4) {};
            \draw[thick,<-] (s3) to node[auto] {$\pi_2$} (s4);
            \draw[thick,->] (s4) to node[auto] {$\pi_3$}(s5);
            \draw[thick,new] (s5) to [out=90,in=90] (s10);
            \draw[thick,->] (s9) to node[auto,swap] {$\pi_1$}(s10);
            \draw[thick,new] (s9) to (a1);
            \draw[thick,new] (a1) to (a2);
            \draw[thick,new] (a2) to (s7);
            \draw[thick,<-] (s7) to node[auto] {$\pi_1$}(s8);
            \draw[thick,new] (s8) to[out=225,in=15] (a3);
            \draw[thick,new] (a3)--(a4);
            \draw[thick,new] (a4) to[out=165,in=305] (s3);
            
            \draw[yshift=-0.7in,text width=6.5cm] (0,0) node
            {
              \textbf{Step 2.} Next, we choose how many edges will go in each
              gap between components.  Each gap must contain at least
              one edge, and we must add a total of $k-i$ edges, giving
              us $\binom{k-i-1}{p-1}$ choices.
              In this example, we have added one edge after $A_1$,
              three after $A_3$, and two after $A_2$.
            };
          \end{scope}
          \begin{scope}[yshift=-4cm,xshift=3.7cm,new/.style={}]
            \path (-1.476,0) node[vert,label=below:{$3$}] (s3) {}
                  -- ++(0.563,0) node[vert,label=below:$4$] (s4) {}
                  -- ++(0.563,0) node[vert,label=below:$5$] (s5) {}
                  -- ++(0.35,0) node[vert,label=below:$10$] (s10) {}
                  -- ++(0.563,0) node[vert,label=below:$9$] (s9){}
                  -- +(-0.1065,0.553) node[vert,label={[new]
                                       above left:$23$}] (a1){}
                  -- +(0.4565,0.553) node[vert,
                            label={[new]above right:$1$}] (a2) {}
                  -- ++(0.35,0) node[vert,label=below:$7$] (s7) {}
                  -- ++(0.563,0) node[vert,label=below:$8$] (s8){}
                  -- (0.492,-0.5) node[vert,label=below:$14$] (a3) {}
                  -- (-0.492,-0.5) node[vert,label=below:$21$] (a4) {};
            \draw[thick,<-] (s3) to node[auto] {$\pi_2$} (s4);
            \draw[thick,->] (s4) to node[auto] {$\pi_3$}(s5);
            \draw[thick,<-] (s5) to [out=90,in=90] node[auto,new]
            {$\pi_1$} (s10);
            \draw[thick,->] (s9) to node[auto,swap] {$\pi_1$}(s10);
              \draw[thick,->] (s9) to (a1);
            
            \draw (s9) +(-0.21,0.32) node[new] {$\pi_2$};
            \draw[thick,->] 
                (a1) to node[auto,new] {$\pi_3$} (a2);
            \draw[thick,<-] (a2) to (s7);
            \draw (s7) +(0.21,0.32) node[new] {$\pi_2$};
            \draw[thick,<-] (s7) to node[auto] {$\pi_1$}(s8);
            \draw[thick,<-] (s8) to[out=225,in=15] node[auto,new] {$\pi_2$}
              (a3);
            \draw[thick,->] (a3)--node[auto,new] {$\pi_1$}(a4);
            \draw[thick,->] (a4) to[out=165,in=305] node[auto,new]{$\pi_1$}
            (s3);
            
            \draw[yshift=-0.7in,text width=6.5cm] (0,0) node
            {
              \textbf{Step 3.} We can choose the new vertices in
              $[n-p-i]_{k-p-i}$ ways, and we can direct and give labels
              to the new edges in at most $(2d-1)^{k-i}$ ways.
            };
          \end{scope}
        \end{tikzpicture}          
         \end{center}
         \caption{Assembling an element $t\in\Cc_i$ that
         overlaps with $s$ at a given subgraph $H$.}
         \label{fig:graphassembly}
       \end{figure}

       All together, there are at most
       $(p-1)!2^{p-1}\binom{k-i-1}{p-1}[n-p-i]_{k-p-i}(2d-1)^{k-i}$
       elements of $\Cc_i$ that overlap with the cycle $s$ at the
       subgraph $H$.  We now calculate the number of different ways to choose
       a subgraph $H$ of $s$ with $i$ edges and $p$ components.
       Suppose $s$ is given as in \eqref{eq:sform}.
       We first choose a vertex $s_j$.  Then, we can
       specify which edges to include in $H$ by giving a sequence
       $a_1,b_1,\ldots,a_p,b_p$ instructing us
       to include the first $a_1$ edges after $s_j$ in $H$,
       then to exclude the next $b_1$, then to include the next $a_2$,
       and so on.  Any sequence for which $a_i$ and $b_i$
       are positive integers, $a_1+\cdots+ a_p=i$,
       and $b_1+\cdots+b_p=k-i$ gives us a valid choice of $i$ edges of $s$
       making up $p$ components.  This counts each subgraph $H$ a total of
       $p$ times, since we could begin with any component of $H$.
       Hence the number of subgraphs $H$ with $i$ edges
       and $p$ components is $(k/p)\binom{i-1}{p-1}\binom{k-i-1}{p-1}$.
       This gives us the bound
       \begin{align*}
         |\Cc_i|&\leq \sum_{p=1}^{i\wedge (k-i)}
              (k/p)\binom{i-1}{p-1}\binom{k-i-1}{p-1}^2
           (p-1)!2^{p-1}[n-p-i]_{k-p-i}(2d-1)^{k-i}.
       \end{align*}
       We apply the bounds $\binom{i-1}{p-1}\leq k^{p-1}/(p-1)!$
       and $\binom{k-i-1}{p-1}\leq (e(k-i-1)/(p-1))^{p-1}$
       to get
       \begin{align*}
         |\Cc_i|
           &\leq k(2d-1)^{k-i}[n-1-i]_{k-1-i}\left(1+
              \sum_{p=2}^{i\wedge (k-i)}\frac{1}{p}
              \left(\frac{2e^2k^3}{(p-1)^2}\right)^{p-1}
              \frac{1}{[n-1-i]_{p-1}}\right).
       \end{align*}
       Since $k<n^{1/6}$,
       the sum in the above equation is bounded by an absolute constant.
       Using the bound $p_t'\leq 1/[n-k]_{k-i}$ for $t\in\Cc_i$,
       we have
       \begin{align*}
         \sum_{t\in\Cc_i}p_t'= O\left(
         \frac{k(2d-1)^{k-i}}{n}
         \right)
       \end{align*}
       and
       \begin{align*}
         \sum_{i=1}^{k-1}\sum_{t\in\Cc_i}p_t'= O\left(
           \frac{k(2d-1)^{k-1}}{n}
         \right).
       \end{align*}
       These estimates, along with
       $p_s\leq 1/[n]_k$, complete the proof.
    \end{proof}

    All that remains now is to apply this lemma
    to finish
    the proof of Theorem~\ref{thm:poiquant}.
    First, consider the case where $k\geq n^{1/6}$.
    Then $k(2d-1)^k/n>1$
    for sufficiently large values of $n$ (regardless
    of $d$), in which case the theorem is trivially
    satisfied.  By choosing $C_0$ large enough,
    it holds for all $n$ with $k\geq n^{1/6}$.
    
    When $k<n^{1/6}$, we
      apply Lemma~\ref{prop:coupling} and \eqref{eq:steinbounds}
      to \eqref{eq:finalbound}
      to get
      \begin{align*}
        |\E[\lambda g(C+1)-C g(C)]|
        &=\frac{\Delta g}{[n]_k}|\Cc|O\left(\frac{k(2d-1)^k}{n}\right)
        +O\left(\frac{k^{3/2}(2d-1)^{k/2}}{n}\right)\\
        &=O\left(\frac{k(2d-1)^{k}}{n}\right)+
        O\left(\frac{k^{3/2}(2d-1)^{k/2}}{n}\right)
      \end{align*}
      The first term is larger 
      than the second for all but finitely many
      pairs $(k,d)$ with $d\geq 2$.
      Hence
      there exists
      $C_0$ large enough that for all $n$, $k$, and
      $d\geq 2$,
      \begin{align*}
        |\E[\lambda g(C+1)-C g(C)]|
        &\leq \frac{C_0k(2d-1)^k}{n}.\qedhere
      \end{align*}
    \end{proof}

    We will need a multivariate version of this theorem as well.
    Define $(\Cy{k};\ k\geq 1)$ to be 
    independent
    Poisson random variables, with $\Cy{k}$ having mean $a(d,k)/2k$.
    Let $\dtv(X,Y)$ denote the total variation distance between
    the laws of $X$ and $Y$. 
    \begin{thm}\label{thm:multipoiquant}
      There is a constant
      $C_2$ such that
      for all $n$, $k$, and $d\geq 2$,
      \begin{align*}
        \dtv\left(\big(C_1^{(n)}, 
         \ldots,C_r^{(n)}\big),\ \big(\Cy{1},\ldots,\Cy{r}\big)\right)\leq
       \frac{C_2(2d-1)^{2r}}{n}.
      \end{align*}
    \end{thm}
    Our proof will be very similar to the single variable case
    above, except that we use Stein's method for Poisson process
    approximation (see \cite[Section~10.3]{BHJ}).
    Let $\lambda_k=a(d,k)/2k$, and 
    let        
    $e_i\in\ZZ_+^r$ be the vector with $i$th entry one and all
    other entries zero.
    Define the operator $\Aa$ by
    \begin{align*}
      \Aa h(x) = \sum_{k=1}^r\lambda_k\big(h(x+e_k)-h(x)\big)
        +\sum_{k=1}^rx_k\big(h(x-e_k)-h(x)\big)
    \end{align*}
    for any $h\colon \ZZ_+^r\to \RR$ and $x\in\ZZ_+^r$.
    We now describe the function that plays a role analogous
    to $g$ in the single variable case.
    \begin{lemma}\label{lem:mvstein}
      For any set $A\subset\ZZ_+^r$, there is a function
      $h\colon \ZZ_+^r\to\RR$ such that
      \begin{align*} 
        \Aa h(x)=1_{x\in A}- 
        \P\big[\big(\Cy{1},\ldots,\Cy{r}\big)\in A\big].
      \end{align*}
      This function $h$ has the following properties:
      \begin{align}
        &\sup_{\substack{x\in\ZZ_+^r\\ 1\leq k\leq r}}|h(x+e_k)-h(x)|\leq 1,
          \label{eq:mvsteinbound1}\\
        &\sup_{\substack{x\in\ZZ_+^r\\ 1\leq j,k\leq r}}|h(x+e_j+e_k) - h(x+e_j)
         +h(x)- h(x+e_k)|\leq 1.\label{eq:mvsteinbound2}
      \end{align}
    \end{lemma}
    \begin{proof}
      This follows from Proposition~10.1.2 and Lemma~10.1.3 in
      \cite{BHJ} as applied to
      a point process on a space with $r$ elements.
    \end{proof}
    
    Our goal is thus to bound $\E\big[\Aa h\big(C_1^{(n)},
    \ldots,C_r^{(n)}\big)\big]$ for any function $h$
    as in Lemma~\ref{lem:mvstein}.  We will
    abbreviate this vector to $\C=(C_1^{(n)},
    \ldots,C_r^{(n)})$.  The set of equivalence class of closed
    trails of length $k$, which we previously denoted $\Cc$,
    we will now call $\Cc^k$.
    \begin{align*}
      \E[\Aa h(\C)]
        &=\sum_{k=1}^r\sum_{s\in\Cc^k}\left(\frac{1}{[n]_k}\E[h(\C+e_k)-
        h(\C)]
          +\E\big[F_s\big(h(\C-e_k)-h(\C)\big)\big]\right)\\
        &=\sum_{k=1}^r\sum_{s\in\Cc^k}\left(\frac{1}{[n]_k}\E[h(\C+e_k)-
        h(\C)]
          +p_s\E\big[h(\C-e_k)-h(\C)\ \big|\ F_s=1\big]\right).
    \end{align*}
    For every $s\in\Cc^k$,
    we will construct on the same probability space as $\C$
    a random variable
    $\Y_s$ such that
    \begin{align}
      \Y_s \eqd \bigg(C_1^{(n)},\ldots,\ C_{k-1}^{(n)},\ 
      \sum_{\substack{t\in\Cc^k\\t\neq s}}F_t,\ C_{k+1}^{(n)},\ldots,\ 
         C_r^{(n)}\bigg)\ \bigg|\ F_s=1.\label{eq:coupledist}
    \end{align}
    Then
    \begin{align*}
      \big|\E[\Aa h(\C)]\big|&=
        \left|\sum_{k=1}^r\sum_{s\in\Cc^k}\left(\frac{1}{[n]_k}\E[h(\C+e_k)-h(\C)]
          +p_s\E[h(\Y_s)-h(\Y_s+e_k)]\right)\right|\\
          &\leq \sum_{k=1}^r
              \sum_{s\in\Cc^k}\frac{1}{[n]_k}\E\big|h(\C+e_k)-h(\C) + 
            h(\Y_s)-h(\Y_s+e_k)\big| \\
            &\phantom{\leq\quad}+
              \sum_{k=1}^r\sum_{s\in\Cc^k}\left|\frac{1}{[n]_k}-p_s\right|
            \E\big|h(\Y_s)-h(\Y_s+e_k)\big|.
    \end{align*}
    By \eqref{eq:mvsteinbound1} and \eqref{eq:mvsteinbound2}, respectively,
    \begin{align*}
      \big|h(\Y_s)-h(\Y+e_k)\big|&\leq 1,\\
      \big|h(\C+e_k)-h(\C) + 
            h(\Y_s)-h(\Y+e_k)\big|&\leq \norm{\C-\Y_s}_1.
    \end{align*}
    Hence
    \begin{align*}
      \big|\E[\Aa h(\C)]\big|&\leq
         \sum_{k=1}^r
              \sum_{s\in\Cc^k}\frac{1}{[n]_k}\E\norm{\C-\Y_s}_1
           +\sum_{k=1}^r\sum_{s\in\Cc^k}\left|\frac{1}{[n]_k}-p_s\right|\\
         &\leq
         \sum_{k=1}^r
              \sum_{s\in\Cc^k}\frac{1}{[n]_k}\E\norm{\C-\Y_s}_1
           +\sum_{k=1}^r|\Cc^k|\frac{k^2}{2n[n]_k}\\
         &= \sum_{k=1}^r
              \sum_{s\in\Cc^k}\frac{1}{[n]_k}\E\norm{\C-\Y_s}_1
           +O\left(\frac{r(2d-1)^r}{n}\right).
    \end{align*}
    Theorem~\ref{thm:multipoiquant}
    then follows
    from the following lemma:
    \begin{lemma}\label{prop:multicoupling}
      There exists an absolute constant
      $C_3$ with the following property.
      For any $1\leq k\leq r$
      and $s\in\Cc^k$,
      let $\Y_s$ 
      be distributed as in \eqref{eq:coupledist}.
      There is a coupling of $\C$ and $\Y_s$ such that
      for all $n$, $k$, and $d\geq 2$ satisfying $k<n^{1/6}$,
      \begin{align}
        \E\norm{\C-\Y_s}_1\leq\frac{C_3r(2d-1)^r}{n}
          \label{eq:multicouplebound}
      \end{align}
    \end{lemma}
    \begin{proof}
      This proof is nearly identical to that
       of Lemma~\ref{prop:coupling}. 
      We construct as before the graph $G_n'$ and the random variables $F_t'$
      for $t\in\Cc^i$, $1\leq i\leq r$.  Then $\Y_s$ can be defined in
      the natural way as
      \begin{align*}
        \Y_s = \bigg(\sum_{t\in\Cc^1}F'_t,\ \ldots,\ \sum_{t\in\Cc^{k-1}}F'_t,\ 
          \sum_{\substack{t\in\Cc^k\\t\neq s}}F'_t,\ \sum_{t\in\Cc^{k+1}}F'_t,\ \ldots,
            \sum_{t\in\Cc^{r}}F'_t\bigg).
      \end{align*}
      We define $\Cc^i_{-1},\ldots,\Cc^i_{(i-1)\wedge k}$ as before, and it remains
      true that $F'_t\geq F_t$ if $t\in\Cc^i_j$ for $j\geq 0$,
      and $F'_t=0$ if $t\in\Cc^i_{-1}$.  Doing the calculation
      just as in \eqref{eq:mainbound},
      \begin{align*}
        \E\norm{\C-\Y_s}_1 &\leq\sum_{i=1}^r \left(
          \sum_{t\in \Cc_{-1}^i}p_t + \sum_{t\in\Cc_0^i}(p_t'-p_t)+
          \sum_{j=1}^{(i-1)\wedge k}\sum_{t\in\Cc_j^i}p_t'
        \right) + p_s.
      \end{align*}
      Nearly identical calculations as in Lemma~\ref{prop:coupling}
      show that 
      \begin{align*}
        \sum_{t\in \Cc_{-1}^i}p_t&=O\left(\frac{k(2d-1)^{i-1}}{n}\right),\\
        \sum_{t\in\Cc_0^i}(p_t'-p_t)&=O\left(\frac{i(2d-1)^i}{n}\right),\\
        \sum_{t\in\Cc_j^i}p_t'&=O\left(\frac{k(2d-1)^{i-j}}{n}\right),
      \end{align*}
      which completes the proof.
    \end{proof}

  \renewcommand{\Y}[2][\infty]{Y_{#2}^{(#1)}}
  \subsection{Non-backtracking walks in random regular graphs}
  \label{sec:nbw}
  We now seek to transfer our results on cycles to closed 
  non-backtracking walks.  Note that we consider $G_n$ as an undirected
  graph when we discuss walks on it.  A non-backtracking walk is one
  that begins and ends at the same vertex, and that never follows
  an edge and immediately follows that same edge backwards.
  Let $\NBW[n]{k}$ denote the number of closed non-backtracking walks of
  length $k$ on $G_n$.  
  
  If the last step of a closed non-backtracking walk is anything other
  than the reverse of the first step, we say that the walk is \emph{cyclically
  non-backtracking}.  Cyclically non-backtracking walks on $G_n$
  are exactly the closed non-backtracking walks whose words
  are cyclically reduced.
  Cyclically non-backtracking walks are easier to analyze than plain
  non-backtracking
  walks because every cyclic and inverted cyclic shift of a cyclically
  non-backtracking walk remains cyclically non-backtracking.
  Let $\CNBW[n]{k}$ denote the number of closed cyclically
  non-backtracking walks of length $k$ on $G_n$.
  
  These notions sometimes go by different
  names.  In \cite{friedmanalon}, non-backtracking walks are called
  irreducible, and $\NBW[n]{k}$ is called
  $\mathrm{IrredTr}_k(G)$.  Cyclically non-backtracking walks
  are called strongly irreducible, and $\CNBW[n]{k}$
  is called $\mathrm{SIT}_k(G)$.
    \begin{figure}
      \begin{center}
        \begin{tikzpicture}[scale=1.8,vert/.style={circle,fill,inner sep=0,
              minimum size=0.15cm,draw},>=stealth]
            \draw[thick] (0,0) node[vert,label=right:1](s0) {}
                  -- ++(0,1cm) node[vert,label=below right:2](s1) {}
                  -- ++(0.707cm,0.707cm) node[vert,label=right:3](s2) {}
                  -- ++(-0.707cm,0.707cm) node[vert,label=right:4](s3) {}
                  -- ++(-0.707cm,-0.707cm) node[vert,label=left:5](s4) {}
                  --(s1);
        \end{tikzpicture}
      \end{center}
      \caption{The walk $1\to 2\to 3\to 4\to 5\to 2\to 1$ is non-backtracking,
      but not cyclically non-backtracking. Note that such walks have a
    ``lollipop'' shape.}
      \label{fig:cnbw}
    \end{figure}
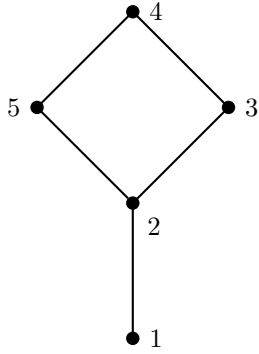

  
  Recall that $(\Cy{k};\ k\geq 1)$ are
  independent
  Poisson random variables, with $\Cy{k}$ having mean $a(d,k)/2k$.
  Define
  \begin{align*}
    \CNBW{k} = \sum_{j|k} 2j\Cy{j}.
  \end{align*}
  For any cycle in $G_n$ of length $j|k$, we obtain $2j$ non-backtracking walks
  of length $k$ by choosing a starting point and direction and then
  walking around the cycle repeatedly.
  We start by decomposing $\CNBW[n]{k}$ into these walks
  plus the remaining ``bad'' walks that are not
  repeated cycles.  We denote these as $\BadW[n]{k}$, giving us
    \begin{align}
      \CNBW[n]{k}&= \sum_{j|k}2j\Cy[n]{j} + B_k^{(n)}.\label{eq:goodbad}
    \end{align}
  The results of Section~\ref{sec:cycles} give us a good understanding
  of $\Cy[n]{k}$.  Our goal now is to analyze $B_k^{(n)}$.
  Specifically, we will show that in the
  right asymptotic regime, it is likely to be zero,
  implying that $\CNBW[n]{k}$ will converge
  to $\CNBW{k}$.
  We start with
   a more precise version of Lemma~\ref{lem:10}.
  \begin{lemma}\label{lem:10ext}
    With the setup of Lemma~\ref{lem:10},
    suppose that $\Gamma$ has $k$ vertices and 
    $e$ edges, with $e>k$.
    Then for all $n>e$,
    \begin{align*}
      \E\big[X_{\Gamma}^{(n)}\big]\leq \frac{1}{[n-k]_{e-k}}
    \end{align*}
  \end{lemma}
  \begin{proof}
    This is apparent from \eqref{eq:10}.
  \end{proof}    
  \begin{prop}\label{prop:badwalks}
    For all $n\geq 2k$,
    \begin{align*}
      \E\big[B_k^{(n)}\big]
      &\leq \sum_{i=1}^{k-1}\frac{a(d,k)k^{2i+2}}{[n-k]_i}.
    \end{align*}
  \end{prop}
  \begin{proof}
    Any closed cyclically non-backtracking walk can be thought of
    as a trail, with repeated vertices in the trail now allowed.
    Such a walk
    is counted by $B_k^{(n)}$ if and only if the graph of its category 
    has more edges than vertices.  Let $\Gg_d$ consist
    of all graphs of categories of a closed trail of length $k$ 
    that have more edges than vertices.  Then 
    \begin{align*}
      B_k^{(n)} = \sum_{\Gamma\in\Gg_d}X_\Gamma^{(n)},
    \end{align*}
    using the notation of Section~\ref{sec:cycles}.
    To use Lemma~\ref{lem:10ext}, we classify the graphs in
    $\Gg_d$ according how to many more edges than vertices they contain:
    \begin{align*}
      \E\big[B_k^{(n)}\big] &\leq \sum_{i=1}^{\infty}
      \big|\{\Gamma\in\Gg_d:\text{$\Gamma$ has exactly
      $i$ more edges than vertices}\}\big|\frac{1}{[n-k]_i}.
    \end{align*}
    A graph in $\Gg_d$ has at most $k$ edges, so the
    terms with $i\geq k$ in this sum are zero.
    By Lemma~18 in \cite{LP}, for each word $w\in\Ww$, the
    number of graphs in $\Gg_d$ with word $w$ and with
    $i$ more edges than vertices is at most
    $k^{2i+2}$, completing the proof.
  \end{proof}
  It is worth noting that this proposition
  fails if the word ``cyclically'' is removed from the definition
  of $B_k^{(n)}$.  The problem is that
  walks that are non-backtracking but not cyclically non-backtracking
  can have as many vertices as edges.
  
  \begin{cor}\label{cor:badwalks}
    There is an absolute constant $C_5$ such that
    for all $n$, $r$, and $d\geq 2$,
    \begin{align*}
      \P[\BadW[n]{k} > 0\text{ for some $k\leq r$}]&\leq
        \frac{C_5r^4(2d-1)^{r}}{n}.
    \end{align*}
  \end{cor}
  \begin{proof}
  Bounding the expression from Proposition~\ref{prop:badwalks} by a geometric
    series,
    \begin{align*}
      \E\big[B_{r}^{(n)}\big] &\leq
        \frac{a(d,r)r^4}{n-r}
        \frac{n-2r}{n-2r-r^2}.
    \end{align*}
    If $r\geq n^{1/4}$, then $r^4(2d-1)^r/n>1$, and 
    the corollary is trivially true for any
    $C_5\geq 1$.  Thus we may assume that $r<n^{1/4}$.
    In this case, the expression $(n-2r)/(n-2r-r^2)$ is
    bounded by an absolute constant.  This and \eqref{eq:adkbounds}
    imply that for some constant $C_4$,
    \begin{align*}
            \E\big[B_{r}^{(n)}\big] &\leq \frac{C_4r^4(2d-1)^{r}}{n}.
    \end{align*}
    Since $\BadW[n]{k}$ is integer-valued,
  \begin{align*}
    \P[\BadW[n]{k} > 0\text{ for some $k\leq r_n$}]
      &\leq \sum_{k=1}^{r}\P[\BadW[n]{k} > 0]\leq \sum_{k=1}^{r}\E[\BadW[n]{k}]\\
      &\leq \sum_{k=1}^{r}\frac{C_4k^4(2d-1)^k}{n}
      \leq \frac{C_5r^4(2d-1)^{r}}{n}
  \end{align*}
  for some choice of the constant $C_5$.
  \end{proof}
  


  The following fact follows directly from the definition of total variation distance,
   and we omit its proof.
  \begin{lemma}\label{lem:dtv}
    Let $X$ and $Y$ be random variables on a metric space $S$,
    and let $T$ be any metric space.
    For any measurable $f\colon S\to T$,
    \begin{align*}
      \dtv(f(X),f(Y))\leq \dtv(X,Y).
    \end{align*}
  \end{lemma}
    
  It is now straightforward to give a result on
  non-backtracking walks analogous to Theorem~\ref{thm:multipoiquant}.
  \begin{prop}\label{thm:CNBWquant} 
    There is a constant $C_6$ such that for all
    $n$, $r$, and $d\geq 2$,
    \begin{align*}
      \dtv\left( \big( \CNBW[n]{k};\ 1\leq k\leq r  \big), 
        \big( \CNBW[\infty]{k};\ 1\leq k\leq r   \big)\right)
        \leq \frac{C_6(2d-1)^{2r}}{n}.
    \end{align*}
  \end{prop}
  \begin{proof}
  We start by recalling the decomposition of
  $\CNBW[n]{k}$ into good and bad walks given in \eqref{eq:goodbad}.
  Let $\GoodW[n]{k}=\sum_{j|k}2j\Cy[n]{j}$, so that
  $\CNBW[n]{k}=\GoodW[n]{k}+\BadW[n]{k}$.
  By Lemma~\ref{lem:dtv} and Theorem~\ref{thm:multipoiquant},
  \begin{align}
    \dtv\left( \big( \GoodW[n]{k};\ 1\leq k\leq r  \big), 
    \big( \CNBW[\infty]{k};\ 1\leq k\leq r   \big)  \right)
      &\leq \dtv\left(\big(\Cy[n]{k};\ 1\leq k\leq r\big),\ 
        \big(\Cy{k};\ 1\leq k\leq r\big)\right)\nonumber\\
      &\leq \frac{C_2(2d-1)^{2r}}{n}.\label{eq:gwbound}
  \end{align}
  Then for any $A\subset\ZZ_+^r$,
  \begin{align*}
    &\P\left[\big( \CNBW[n]{k};\ 1\leq k\leq r  \big)\in A
    \right]-\P\left[\big( \CNBW{k};\ 1\leq k\leq r  \big)\in A
    \right]  \\\quad&\leq\P\left[\big( \GoodW[n]{k};\ 1\leq k\leq r  \big)\in A
    \right]+\P\left[\bigcup_{k=1}^r\big\{\BadW[n]{k}>0\big\}
    \right]-\P\left[\big( \CNBW{k};\ 1\leq k\leq r  \big)\in A
    \right]\\\quad&\leq\frac{C_2(2d-1)^{2r}}{n}+\frac{C_5r^4(2d-1)^r}{n}
  \end{align*}
  by \eqref{eq:gwbound} and Corollary~\ref{cor:badwalks}.
  For any $n$ and $d$, since $d \geq 2$ and thus $2d-1 \geq 3$, 
  the first term is larger
      than the second for all but at most a finite number of $r$s,
      bounded independently of $n$ and $d$. Therefore there exists
      a constant $C_6$ satisfying the conditions of the theorem.
  \end{proof}
  \begin{cor}\label{cor:dfixed}
    For any fixed $r$ and $d\geq 2$,
    \begin{align*}
      (\CNBW[n]{1},\ldots,\CNBW[n]{r})\toL
      (\CNBW{1},\ldots,\CNBW{r})
    \end{align*}
    as $n\to\infty$.
  \end{cor}

  To achieve a version of the above
  corollary that holds
  when $d$ grows, we need to center and scale our
  random variables $\CNBW[n]{k}$.
  \begin{prop}\label{thm:dgrows}  
    Let $r$ be fixed, and suppose that $d=d(n)\to\infty$ as $n\to\infty$,
    and that $(2d-1)^{2r}=o(n)$.
    Let $\sCNBW[n]{k}=(2d-1)^{-k/2}
      (\CNBW[n]{k}-\E[\CNBW{k}])$.
   Let $Z_1,\ldots,Z_r$ be independent normal random variables
   with $\E Z_k=0$ and $\E Z_k^2 = 2k$.
    Then as $n\to\infty$,
    \begin{align*}
      \big(\sCNBW[n]{1},\ 
       \ldots,\ 
      \sCNBW[n]{
      r}\big)\toL (Z_1,\ldots,Z_r).
    \end{align*}
  \end{prop}
  \begin{proof}
    Let $\X[n]{k}=(2d-1)^{-k/2}
      (\CNBW{k}-\E[\CNBW{k}])$.
    We note that $\CNBW{k}$ depends on $d$ (and hence on $n$),
    although we have suppressed this dependence from the notation.
    By Proposition~\ref{thm:CNBWquant} and Lemma~\ref{lem:dtv},
    the total variation distance
    between the laws of $\big(\sCNBW[n]{k};\ 1\leq k\leq r\big)$
    and $\big(\X[n]{k};\ 1\leq k\leq r\big)$ converges to zero
    as $n\to\infty$.  Hence
    it suffices to show that $\big(\X[n]{k};\ 1\leq k\leq r\big)\toL
    (Z_1,\ldots,Z_r)$ as $n\to\infty$. 
    
    Let $\lambda_k=a(d,k)/2k$ as in Theorem~\ref{thm:multipoiquant}.
    We can write $\X[n]{k}$ as
      \begin{align*}
        \X[n]{k}&= 2k(2d-1)^{-k/2}\big(\Cy{k}- \lambda_k\big)+
          (2d-1)^{-k/2}\sum_{\substack{j|k\\j<k}}\big(2j\Cy{j} 
          -a(d,j)\big).
      \end{align*}
      Using \eqref{eq:adkbounds}, it is 
      a straightforward calculation
      to show that as $n\to\infty$,
      \begin{align*}
        \left(2k(2d-1)^{-k/2}\big(\Cy{k}- \lambda_k\big);\ 1\leq k\leq r\right)
        \toL (Z_1,\ldots,Z_r).  
      \end{align*}
      Hence we need only show that for all $k\leq r$,
      \begin{align*}
        (2d-1)^{-k/2}\sum_{\substack{j|k\\j<k}}\big(2j\Cy[n]{j} 
           - a(d,j)\big)
           \toPr 0.
      \end{align*}
      We calculate
      \begin{align*}
        \var\bigg[(2d-1)^{-k/2}\sum_{\substack{j|k\\j<k}}\big(2j\Cy[n]{j} 
           - a(d,j)\big)\bigg] &= (2d-1)^{-k}\sum_{\substack{j|k\\j<k}}ja(d,j),
      \end{align*}
      and the statement follows by \eqref{eq:adkbounds} and Chebyshev's inequality.
      \end{proof}
The remaining results in this section refer to the weak convergence
set-up in Section~\ref{sec:weakconvergence}.

                 
\begin{thm}\label{thm:dfixedtight}
  Suppose that $d$ is fixed, that $r_n\to\infty$, and that
  \begin{align}
    (2d-1)^{2r_n} &= o(n)\label{eq:2d}.
  \end{align}
  Let 
  \[
  \Theta_k=\E\big[\CNBW{k}\big]^2=\sum_{j|k} 2j a(d,j) + \left( \sum_{j|k} a(d,j)  \right)^2.
  \]
  Let $(b_k)_{k\in\NN}$ be any fixed positive summable sequence.
  Define the weights of Section~\ref{sec:weakconvergence} by setting 
  \[
  \omega_k=b_k/\Theta_k, \quad k \in \NN.
  \]
  Let $P_n$ be the law of the sequence $(\CNBW[n]{1},\ldots,\CNBW[n]{r_n},0,0,
  \ldots)$.  Then $\{P_n\}$, considered as a sequence in $\PP(X)$, converges
  weakly
  to the law of the random vector $\big(\CNBW{k};\, k\in\NN\big)$.
\end{thm}

\begin{proof} We first claim that the random vector $\left( \CNBW{k};\, k \in \NN  \right)$ almost surely lies in $\ltwo(\weight)$. This follows by a deliberate choice of $\weight$:  
\[ 
\E \sum_{k=1}^\infty \left(\CNBW{k}\right)^2 \omega_k = \sum_{k=1}^\infty \Theta_k \omega_k = \sum_{k=1}^\infty b_k < \infty,
\]
which proves finiteness almost surely. The computation of $\Theta_k$ is straightforward.

  By Corollary~\ref{cor:dfixed}, we know that all subsequential
  weak limits of $P_n$ have the same finite-dimensional distributions
  as $\big(\CNBW{k}; k\in\NN\big)$,
  and by Lemma~\ref{lem:fd-characterize}, they are in fact identical
  to the law of $\big(\CNBW{k}; k\in\NN\big)$.
  Thus it suffices to show that $\{P_1,P_2,\ldots\}$
  is tight. To do this we will apply Lemma \ref{lemma:hilbertcube} by choosing a suitable infinite cube.

  In other words, we must show that given any $\epsilon >0$, there exists an element $\underline{a}=(a_m)_{m\in \NN} \in \ltwo(\weight)$ such that
    \begin{align}
    \sup_n\P\left[\cup_{k=1}^{r_n}
      \left\{ \CNBW[n]{k}> a_k  \right\} \right]<\eps.\label{eq:tightness}
  \end{align}
  In fact, our choice of $\underline{a}$ is 
  \[
  a_k = (\alpha +2) \E\left( \CNBW[\infty]{k}\right)= (\alpha + 2) \sum_{j|k} a(d,j),
  \]
  for some positive $\alpha$ determined by $\epsilon$. Note that, by an obvious calculation, $\underline{a}\in \ltwo(\weight)$. 
  
  By Proposition~\ref{thm:CNBWquant},  for any $\eta>0$,
  \begin{align}
    \P\left[\cup_{k=1}^{r_n} \left\{ \CNBW[n]{k}> a_k  \right\} \right]\le \P\left[
      \cup_{k=1}^{r_n}\left\{ \CNBW[\infty]{k}> a_k  \right\}  \right]+\eta\label{eq:tbound}
  \end{align}
  for all sufficiently large $n$.
   Now, we apply the union bound 
  \eq\label{unionbnd}
\sup_n  \P\left[\cup_{k=1}^{r_n} \left\{ \CNBW[\infty]{k}> a_k  \right\} \right] \le \sum_{k=1}^\infty \P\left[  \CNBW[\infty]{k}  > a_k\right]
  \en
and bound the right side by a simple large deviation estimate.
 
We start with the decomposition
\eq\label{cnbwdecomp}
\CNBW[\infty]{k} = \sum_{j|k} 2j \Cy{j}, 
\en
where $\{\Cy{j}\}$ are independent Poisson random variables with mean $a(d,j)/2j$.
Thus, for any $\lambda >0$, the exponential moments are easy to derive: 
\[
\begin{split}
\E\left(  e^{\lambda \CNBW[\infty]{k}}  \right) &= \prod_{j|k} E\left( e^{\lambda 2j \Cy{j} }  \right)= \prod_{j|k} \exp\left\{ \frac{a(d,j)}{2j}\left( e^{2\lambda j} - 1 \right) \right\}\\
&= \exp\left[  \sum_{j|k} a(d,j)\frac{e^{2\lambda j }-1}{2j}   \right].
\end{split}
\] 
Hence, by Markov's inequality, we get
\[
\begin{split}
\P\left(  \CNBW[\infty]{k} > a_k  \right) &\le e^{-\lambda a_k} \E \left(  e^{\lambda \CNBW[\infty]{k}}  \right)\\
&\le \exp\left[  \sum_{j|k} a(d,j) \left( \frac{e^{2\lambda j } -1 }{2j} - (\alpha+2)\lambda \right) \right].
\end{split}
\] 
An easy analysis shows that if $\lambda= \log 2/ (2k)$, one must have
\[
\frac{e^{2\lambda j}-1}{2j} < 2\lambda, \quad \text{for all $j\le k$}. 
\] 
Hence,
\[
\P\left(  \CNBW[\infty]{k} > a_k  \right) \le \exp\left[  -\frac{\alpha \log 2}{2k} \sum_{j|k} a(d,j)  \right] \le 2^{  -\alpha (2d-1)^{k}/2k}. 
\] 
 
The above expression is clearly summable in $k$, and thus from \eqref{unionbnd} we get 
 \[
 \sup_n  \P\left[\cup_{k=1}^{r_n} \left\{ \CNBW[\infty]{k}> a_k  \right\} \right] \le \sum_{k=1}^\infty 2^{  -\alpha (2d-1)^{k}/2k}.
 \]
The right side can be made as small as we want by choosing a large enough $\alpha$. This is enough to establish \eqref{eq:tightness}.
\end{proof}

We now prove a corresponding theorem when $d$ is growing with $n$. 
Let $\mu_k(d)$ denote $\E\big[\CNBW{k}\big]$ emphasizing its dependence on $d$. We define
\begin{align}\label{eq:whatisntilde}
  \Nc[n]{k} &= (2d-1)^{-k/2}\big(\CNBW[n]{k}-\mu_k(d)\big).
\end{align}

\begin{thm}\label{thm:dgrowstight}
  Suppose that $d=d(n)\to\infty$ and $r_n\to\infty$ as $n\to\infty$.
  Suppose that
  \begin{align*}
    (2d-1)^{2r_n} &= o(n).
  \end{align*}
  We define the weights $\weight$ by setting $\omega_k=b_k/(k^2\log k)$, where
  $(b_k)_{k\in\NN}$ is any fixed positive summable sequence.
  Let $P_n$ be the law of the sequence $(\Nc[n]{1},\ldots,\Nc[n]{r_n},0,0,
  \ldots)$.
  Let $Z_1,Z_2,\ldots$ be independent normal 
  random variables with $\E Z_k=0$ and $\E Z_k^2=2k$.  
  Then $P_n$, considered as an element of $\PP(X)$, converges
  weakly to the law of the random vector $(Z_k; k\in\NN)$.
  \end{thm}

To proceed with the proof we will need a lemma on measure concentration. We will use a result on modified logarithmic Sobolev inequality that can be found in the Berlin notes by Ledoux \cite{Led}. For the convenience of the reader we reproduce (a slight modification of) the statement of Theorem 5.5 in \cite[page 71]{Led} for a joint product measure. Please note that although the statement of Theorem 5.5 is written for an iid product measure, its proof goes through even when the coordinate laws are different (but independent). In fact, the crucial step is the tensorization of entropy (\cite[Proposition 2.2]{Led}), which is generally true. 

\begin{lemma}\label{lemma:modlogsobo}
For $n\in \mathbb{N}$, let $\mu_1, \mu_2, \ldots, \mu_n$ be $n$ probability measures on $\NN$. For functions $f$ on $\NN$, define $Df(x)= f(x+1) - f(x)$ to be the discrete derivative. Define the entropy of $f$ under $\mu_i$ by
  \[
  \Ent_{\mu_i}(f)= \E_{\mu_i}\left( f \log f \right) - \E_{\mu_i} (f) \log \E_{\mu_i}\left( f \right). 
  \]
Assume that there exist two positive constants $c$ and $d$ such that for every $f$ on $\NN$ such that $\sup_x\abs{Df} \le \lambda$, one has
\[
 \Ent_{\mu_i}\left( e^f\right) \le c e^{d\lambda} \E_{\mu_i}\left( \abs{D f}^2 e^f \right), \quad \text{as functions of $\lambda$}.
\]

Let $\mu$ denote the product measure of the $\mu_i$'s. Let $F$ be a function on $\NN^n$ such that for every $x\in \NN^n$, 
\[
\sum_{i=1}^n \abs{F(x+e_i) - F(x)}^2 \le \alpha^2,\quad \text{and}\quad \max_{1\le i \le n} \abs{F(x+e_i) - F(x)} \le \beta.
\]

Then $\E_\mu(\abs{F}) < \infty$ and, for every $r\ge 0$, 
\[
\mu\left( F \ge \E_\mu(F) + r \right) \le \exp\left( -\frac{r}{2d\beta} \log\left(1 + \frac{\beta dr}{4 c\alpha^2} \right)  \right).
\] 
\end{lemma}

\begin{proof}[Proof of Theorem \ref{thm:dgrowstight}] The proof is similar in spirit to the proof of Theorem \ref{thm:dfixedtight}. 
 As in that proof, the limiting measure is supported on $\ltwo(\weight)$.
  By Proposition~\ref{thm:dgrows} and Lemma~\ref{lem:fd-characterize},
  we need only show that the family $\{P_1,P_2,\ldots\}$ is tight.
  As in Theorem~\ref{thm:dfixedtight}, we need to choose a suitable infinite cube. 
  
  Choose $\epsilon >0$. Define
  \begin{align*}
  a_k   &= \alpha k \sqrt{\log k},
  \end{align*}
  for some positive $\alpha > 1$ depending on $\epsilon$. Then $\underline{a} \in \ltwo(\weight)$.

  We need to show that, for a suitable choice of $\alpha$, 
  \begin{align*}
    \sup_n\P\left[\cup_{k=1}^{r_n}\left\{ \abs{\Nc[n]{k}} > a_k  \right\}\right]<\eps.
  \end{align*}
  
  By Lemma~\ref{lem:dtv} and Proposition~\ref{thm:CNBWquant},
  for any $\eta>0$,
  \begin{align}
    \P\left[\cup_{k=1}^{r_n}\left\{ \abs{\Nc[n]{k}} > a_k  \right\}\right]<\P\left[\cup_{k=1}^{r_n}\left\{ \abs{\CNBW[\infty]{k} - \mu_k(d)} > a_k (2d-1)^{k/2}  \right\}\right]+\eta\label{eq:tbounddgrows}
  \end{align}
  for all sufficiently large $n$. 

  Note as before that $\CNBW{k}$ depends on $d$ 
  (and hence on $n$).
  
  Proceeding as before, we need to estimate
  \begin{align*}
 \P\left(  \abs{\CNBW[\infty]{k} - \mu_k(d)} > a_k (2d-1)^{k/2}  \right)
   \end{align*}
  for our choice of $a_k$. 
    
  Let $\Poi(\theta)$ denote as before the Poisson law with mean $\theta$. We will denote expectation with respect to $\Poi(\theta)$ by $\E_{\pi_{\theta}}$. As shown in Corollary 5.3 in \cite[page 69]{Led}, $\Poi(\theta)$ satisfies the following modified logarithmic Sobolev inequality: for any $f$ on $\mathbb{N}$ with strictly positive values
  \eq\label{modlogsobo}
  \Ent_{\pi_{\theta}}(f) \le \theta \E_{\pi_\theta} \left( \frac{1}{f} \abs{Df}^2  \right).  
  \en
  Here $\Ent_{\pi_\theta}(f)$ refers to the entropy of $f$ under $\Poi(\theta)$.
    
  Let now $f$ on $\NN$ satisfy $\sup_{x} \abs{Df(x)} \le \lambda$. By eqn. (5.16) in \cite[page 70]{Led}, \eqref{modlogsobo} implies that $\Poi(\theta)$ satisfies the inequality
  \eq\label{modlogsobo2}
  \Ent_{\pi_\theta}\left( e^f\right) \le C e^{2\lambda} \E_{\pi_\theta}\left( \abs{Df}^2 e^f \right), \quad \text{  for any $C\ge \theta$.}
  \en

  Now fix some $k\in \NN$ and consider the product measure of the random vector $(\Cy{j},\; j|k  )$. 
  Each coordinate satisfies inequality \eqref{modlogsobo2} and one can take the common constant $C$ to be $a(d,k)/2k$. 
    
  We apply Lemma \ref{lemma:modlogsobo} on the function $F(\underline{x}) = \sum_{j|k} 2j x_j$. It is straightforward to see that one can take $\alpha^2=  4k^3$, $\beta=2k$. Thus, we get the following tail estimate for any $r >0$:
  \[
  \P\left(  F > \E(F) + r \right) \le \exp\left( -\frac{r}{8k}\log\left( 1 + \frac{4k r}{4 C 4k^3 } \right)  \right).
  \]
  Replacing $F$ by $-F$ we obtain a two-sided bound
  \[
  \P\left( \abs{F - \E(F)} > r  \right) \le 2\exp\left( -\frac{r}{8k}\log\left( 1 + \frac{4k r}{16 C k^3 } \right)  \right).
  \]
 
 Hence we have shown that for any $r >0$, the following estimate holds
 \[
 \begin{split}
 \P\left(  \abs{\CNBW[\infty]{k} - \mu_k(d)} > r  \right)&\le 2\exp\left( -\frac{r}{8k}\log\left( 1 + \frac{8k^2 r}{16 a(d,k) k^3 } \right)  \right)\\
 &= 2\exp\left( -\frac{r}{8k}\log\left( 1 + \frac{r}{ 2a(d,k) k} \right)  \right).
 \end{split}
 \]
 
 Recall from \eqref{eq:adkbounds} that $a(d,k)\sim (2d-1)^k$. Therefore
 \[
 \P\left(  \abs{\CNBW[\infty]{k} - \mu_k(d)} > a_k (2d-1)^{k/2} \right) \le 2\exp\left( -\frac{a_k (2d-1)^{k/2}}{8k}\log\left( 1 + \frac{a_k}{ 2(2d-1)^{k/2} k} \right)  \right).
 \]
 
 Now, $\log(1+x) \ge x/2$ for all $0\le x\le 1$. Using this simple bound we get that for all $(k,d)$ such that $\alpha \log k \le 2(2d-1)^{k}$, we have
\[
 \P\left(  \abs{\CNBW[\infty]{k} - \mu_k(d)} > a_k (2d-1)^{k/2} \right) \le 2 \exp\left( -\frac{a_k^2}{32 k^2} \right)\le 2\exp\left( -\frac{\alpha^2 k^2 \log k}{32 k^2}\right)=2 k^{-\alpha^2/32}. 
\]
The right side is summable whenever $\alpha^2 > 32$. The rest of the proof follows just as in Theorem~\ref{thm:dfixedtight}.
\end{proof}

\section{Spectral concentration}\label{sec:eval2} 

The problem of estimating the spectral gap of a $d$-regular graph has
been approached primarily in two ways, the method of moments and the
counting method of Kahn and Szemer\'edi, prezented in \cite{FKSz}.  The
method of moments has been developed in the work of Broder and Shamir
\cite{broder_shamir} and very extensively by Friedman \cite{friedmane2}, \cite{friedmanalon}.  In his
work, Friedman, relying on $d$ being fixed independently of
$n$, developed extremely fine control over the magnitude of the second
eigenvalue.  On the other hand in \cite{FKSz}, Kahn and Szemer\'edi
only show that the second largest eigenvalue has magnitude
$O(\sqrt{d}).$  While weaker than Friedman's bound, their techniques
readily extend to the case where $d$ is allowed to grow as a function
of $n$; this observation has been informally made by others, and communicated to
us by Vu and Friedman.  Here we will formalize it,  and present the
Kahn-Szemer\'edi argument in the context of growing $d$ to demonstrate
the method's validity, as well as to develop some handle on the
constants in the bound.

Specifically, we will prove \begin{thm}
\label{thm:eigbound}
For any $m > 0,$ there is a constant $C=C(m)$ and universal constants $K$ and $c$ so that
\[
\prob \left[
\exists i \neq 1 ~:~ |\lambda_i| \geq C \sqrt{d} 
\right] \leq n^{-m} + K\exp(-cn).
\]
Further, the constant $C$ may be taken to be $36000 + 2400m.$
\end{thm}  

In what follows, let $M$ be the adjacency matrix for the $2d$-regular graph $G_n$.
Recall that
  this matrix can be realized by sampling independently and uniformly $d$ permutation matrices $A_1, A_2, \ldots, A_d$ and defining
\[
M = A_1 + A_1^t + A_2 + A_2^t + \cdots + A_d + A_d^t.
\] 
The starting point is the variational characterization of the eigenvalues $\lambda_1 \geq \lambda_2 \geq \cdots \geq \lambda_n$ of $M$, which states that
\[
\max\{ \lambda_2, |\lambda_n| \} = \sup_{\substack{w \perp \one \\ \|w\|=1}} \left|w^t M w\right|.
\]
Additional flexibility is provided by replacing this symmetric version of the Rayleigh quotient by the asymmetric version, 
\[
\sup_{\substack{w,v \perp \one \\ \|v\|=\|w\|=1}} |v^t M w|.
\]

The random variables $v^tMw$, for fixed $w$ and $v$, are substantially more tractable than the supremum.  To be able to work with these random variables instead of the supremum, we will pass to a finite set of vectors which approximate the sphere $\mathcal{S} = \{w \perp \one~:~ \|w\|=1\}.$  More specifically, we will only consider those $w$ and $v$ lying on the subset of the lattice $\mathcal{T}$ defined as
\[
\mathcal{T} := \left\{ \frac{\delta z }{\sqrt{n}} ~:~ z \in \mathbb{Z}^n, \|z\|^2 \leq \frac{n}{\delta^2}, z \perp \one \right\},
\]
for a fixed $\delta > 0.$

Vectors from $\mathcal{T}$ approximate vectors from $\mathcal{S}$ in the sense that every $v \in (1-\delta)\mathcal{S}$ is a convex combination of points in $\mathcal{T}.$ (See Lemma 2.3 of~\cite{FeigeOfek}.)  
Thus  
\[
\frac{1}{(1-\delta)^2} \sup_{\substack{w,v \perp \one \\ \|v\|=\|w\|=1}} \left|[1-\delta]v^t M [1-\delta]w\right| \leq \frac{1}{(1-\delta)^2} \sup_{x,y \in \mathcal{T}} \left|x^tMy\right|.
\]
Furthermore, by a volume argument, it is possible to bound the cardinality of $\mathcal{T}$ as 
\[
\left|\mathcal{T}\right| \left(\frac{\delta}{\sqrt{n}}\right)^n \leq \operatorname{Vol}\left[ x \in \R^n ~:~ \|x\| \leq 1+\tfrac{\delta}{2} \right] = \frac{(1+\tfrac{\delta}{2})^n \sqrt{\pi}^n }{\Gamma(\tfrac{n}{2} + 1)}.
\]
Employing Stirling's approximation, this shows
\[
\left|\mathcal{T}\right| \leq C \left[\frac{(1+\tfrac{\delta}{2})\sqrt{2e\pi}}{\delta} \right]^n.
\]
for some universal constant $C$.

The breakthrough of Kahn and Szemer\'edi was to realize that $x^t M y$
can be controlled by virtue of a split into two types of terms.  If $x^tMy$ is written as a sum
\[
x^t M y = \sum_{\substack{ (u,v) \\ |x_uy_v| < \tfrac{\sqrt{d}}{n}}} x_uM_{uv}y_v + \sum_{\substack{ (u,v) \\ |x_uy_v| \geq \tfrac{\sqrt{d}}{n}}} x_uM_{uv}y_v, 
\]
then the contribution of the first sum turns out to be very nearly its
mean because of the Lipschitz dependence of the sum on the edges of
the graph.   The contribution of the second sum turns out to never be too large for a very different reason: the number of edges between any two sets in the graph is on the same order as its mean.  Following Feige and Ofek, for a fixed pair of vectors $(x,y) \in \mathcal{T}^2,$ define the \emph{light couples} $\mathcal{L} = \mathcal{L}(x,y)$ to be all those ordered pairs $(u,v)$ so that $|x_uy_v| \leq \tfrac{\sqrt{d}}{n},$ and let the \emph{heavy couples} $\mathcal{H}$ be all those pairs that are not light.

\subsection{Controlling the contribution of the light couples.}
Part of the advantage of having selected only the light couples is
that their expected contribution is of the ``correct'' order, as the
lemma below shows. 
\begin{lemma}
\label{lem:lightcoupleexpect}
\[
\left| 
\expect \sum_{ (u,v) \in \mathcal{L}} x_uM_{uv}y_v
\right| \leq 2\sqrt{d}.
\]
\end{lemma}

\begin{proof}

By symmetry, $\expect M_{uv}$ is simply equal to $\frac{2d}{n},$ so that  
\[
\expect \sum_{ \{u,v\} \in \mathcal{L}} x_uM_{uv}y_v
= \frac{2d}{n}\sum_{ \{u,v\} \in \mathcal{L}} x_uy_v.
\]
Because each of $x_u$ and $y_v$ sum to $0,$ the sum over light couples is equal in magnitude to the sum over heavy couples.  Thus, it suffices to estimate
\begin{align*}
\left| \sum_{ \{u,v\} \in \mathcal{H}} x_uy_v \right|
& \leq \sum_{ \{u,v\} \in \mathcal{H}} \left| x_uy_v \right| = \sum_{ \{u,v\} \in \mathcal{H}} \frac{x_u^2y_v^2}{\left|x_uy_v\right|}\\
& \leq \frac{n}{\sqrt {d}}\sum_{ \{u,v\} \in \mathcal{H}}{x_u^2y_v^2},\quad 
\text{by the defining property of heavy couples,}\\
& \leq \frac{n}{\sqrt{d}}.
\end{align*}
In the last step we recall that both $\|x\|, \|y\| \leq 1.$
\end{proof}

To show that not only the expectation, but the sum itself is of the
correct order,  we must prove a 
concentration estimate for this sum.  For technical reasons, it is helpful if we deal with sums over fewer terms.   To this end, define
\[
A = A_1 + A_2 + \cdots + A_d.
\]
In terms of $A$ it is enough to insist that for every $x,y \in \mathcal{T}$ 
\[
\left|\sum_{ (u,v) \in \mathcal{L}} x_uA_{uv}y_v\right| \leq t\sqrt{d}
\]
for then by symmetry,
\[
\left|\sum_{ (u,v) \in \mathcal{L}} x_uM_{uv}y_v\right| \leq 2t\sqrt{d},
\]
for all $x,y \in \mathcal{T}.$  As a further simplification, we will not prove a 
tail estimate
for the whole quantity  $\sum\limits_{ (u,v) \in \mathcal{L}} x_uA_{uv}y_v$; 
instead, fix an arbitrary collection $U$ of vertices of size at most $\lceil \tfrac{n}{2} \rceil.$  Having fixed this collection, we will show a 
tail estimate for $\sum_{ (u,v) \in \mathcal{L} \cap U \times [n] }
x_uA_{uv}y_v.$
This truncation is made to simplify a variance estimate (see \eqref{truncation_need}),
and it might be possible to avoid it entirely.

\begin{thm}
\label{thm:lightlargedeviation}

For every $x,y \in \mathcal{T}$, and every $U \subset [n]$ with $|U| \leq \lceil\tfrac{n}{2}\rceil,$
\[
\prob \left[
\left|\sum_{ (u,v) \in \mathcal{L} \cap U \times [n]} x_uA_{uv}y_v - \expect x_uA_{uv}y_v \right| > t\sqrt{d}
 \right] \leq C_0\exp\left( -\frac{nt^2}{C_1 + C_2t} \right)
\]
for some universal constants $C_0$, $C_1$ and $C_2.$  These constants can be taken as $2,$ $64,$ and $8/3$ respectively. 
\end{thm}

\begin{proof}
Let $\tilde{\mathcal{L}}$ be $\mathcal{L} \cap U \times [n].$  We will estimate tail probabilities for 
$\sum\limits_{ (u,v) \in \tilde{\mathcal{L}} } x_uA_{uv}y_v.$

The main tool needed to establish this result is Freedman's martingale
inequality \cite{Freedman}.  Let $X_1, X_2, \ldots$ be
martingale increments.  Write $\mathscr{F}_k$ for the natural
filtration induced by these increments, and define $V_k = \expect
\left[ X_k^2 ~|~ \mathscr{F}_{k-1} \right].$  If $S_n$ is the partial
sum $S_n = \sum_{i=1}^n X_i$ (with $S_0 = 0$) and $T_n$ is the sum
$T_n = \sum_{i=1}^n V_i$ (with  $T_0 = 0$), then by analogy with the continuous case, one expects $S_n$ to be a Brownian motion at time $T_n$ (a discretization of the bracket process).  The analogy requires, however, that the increments have some {\em a priori} bound.  Namely, if $|X_k| \leq R,$ 
\[
\prob \left[ \exists~n \leq \tau \text{ so that } S_n \geq a \text{ and } T_n \geq b \right] \leq 
2\exp\left( -\frac{a^2/2}{\tfrac{Ra}{3} + b} \right).
\]
\begin{rmk}
The constants quoted here are slightly better than the constants that appear in Freedman's original paper.  This statement of the theorem follows from Proposition 2.1 of~\cite{Freedman} and the calculus lemma
\[
(1+u) \log ( 1 + u) - u \geq \frac{u^2/2}{1+u/3},
\]
for $u \geq 0.$
\end{rmk}

Reorder and relabel the vertices from $U$ as $x_1,x_2,\ldots, x_r,$ with $r \leq \lceil \tfrac n 2 \rceil$ so that $|x_j|$ decreases in $j.$  Order pairs $(i,j) \in [d] \times \{0,1,2,\ldots r\}$ lexicographically, and enumerate $\pi_i(j)$ in this order as $f_1,f_2, \ldots , f_{rd}.$  Define a filtration of $\sigma$-algebras $\{\mathscr{F}_{k}\}_{k=1}^{rd}$ by revealing these pieces of information, i.e. $\mathscr{F}_k = \mathscr{F}_{k-1} \vee \pi(f_k).$ According to this filtration, let \[
S_k = \expect\left[ \sum_{ (u,v) \in \tilde{\mathcal{L}}} x_uA_{uv}y_v \bigg\vert \mathscr{F}_k \right]
\]
define a martingale and let $X_k = X_{(i,j)}$ be the associated martingale increments.

The desired deviation bound can now be cast in terms of $S_k$ as
\begin{align*}
\hspace{1in}&\hspace{-1in}\prob \left[
\left|\sum_{ \tilde{\mathcal{L}} } x_uA_{uv}y_v - \expect x_uA_{uv}y_v\right| \geq t 
\right] \leq \prob \left[ \exists~k \leq rd \text{ so that } |S_k -
  S_0| = |S_k| \geq t \text{ and } T_n \geq b \right] \\
&\leq 2\exp\left( \frac{ -t^2/2}{(\tfrac{R t}{3} + b)}\right),
\end{align*}
provided that $b$ satisfies
\[
\sum_{k=1}^{rd} \expect\left[ X_k^2 ~\big\vert~ \mathscr{F}_{k-1}\right] \leq b.
\]

This reduces the problem to finding suitable $R$ and $b.$  The starting point for finding any such bound is simplifying the expression for the martingale increments $X_{(i,k)}.$   To this end, let $\pi$ be a fixed permutation of $[n],$ and define $\Pi_k$ to be the collection of all permutations that agree with $\pi$ in the first k entries, i.e.
\[
\Pi_k = \{ 
\sigma ~:~ \sigma(i) = \pi(i)~i = 1, 2, \ldots, k
\}.
\]
Further let $T : \Pi_{k-1} \to \Pi_{k}$ be the map which maps a permutation to its nearest neighbor in $\Pi_{k},$ in the sense of transposition distance, i.e.
\[
T[\sigma](i) = \begin{cases}
\pi(k) & i = k \\
\sigma(k) & i = \sigma^{-1}(\pi(k)) \\
\sigma(i) & \text{ else }
\end{cases}.
\]
Note that this map is the identity upon restriction to $\Pi_k.$  Let $\oneL{u}{v}$ be the characteristic function for $(u,v) \in \tilde{\mathcal{L}}.$ In terms of these notation, it is possible to express $X_{(i,k)}$ as 
\[
X_{(i,k)} = \frac{1}{|\Pi_{k-1}|} \sum_{ \tau \in \Pi_{k-1}} \sum_{u \in U} x_u \oneL{u}{T[\tau](u)} y_{T[\tau](u)} - x_u \oneL{u}{\tau(u)} y_{\tau(u)},
\]
where $\pi = \sigma_i,$ and the contributions of the other $\sigma_j$ all cancel.  As $\tau(u) = T[\tau](u)$ except for when $u=k$ or $u = \tau^{-1}( \pi (k)),$ this simplifies to
\begin{align*}
X_{(i,k)} = &\frac{1}{|\Pi_{k-1}|} \sum_{ \tau \in \Pi_{k-1}}  \left
  (x_u \oneL{u}{\pi(k)} y_{\pi(k)} - x_u \oneL{u}{\tau(k)}y_{\tau(k)}
\right . \\
&\hspace{1in} \left . + x_{\tau^{-1}(\pi(k))}
  \oneL{\tau^{-1}(\pi(k))}{\tau(k)}y_{\tau(k)} - x_{\tau^{-1}(\pi(k))}
  \oneL{\tau^{-1}(\pi(k))}{\pi(k)}y_{\pi(k)} \right ).
\end{align*}
This can be recast probabilistically.  Define two random variables $v$ and $u$ as
\begin{align*}
v &\sim \operatorname{Unif} \left\{ [n] \setminus \pi[k] \right\} ~,\\
u &\sim \operatorname{Unif} \left\{ [n] \setminus [k] \right\},
\end{align*}
(where $[n] = \{1,2,\ldots, n\}$) so that
\begin{eqnarray} \label{starstar}
\tfrac{n-k+1}{n-k}X_k & = & \expec \big [ x_k \oneL{k}{v} y_v
- x_k \oneL{k}{\pi(k)} y_{\pi(k)}
+ x_u \oneL{u}{\pi(k)} y_{\pi(k)}
- x_u \oneL{u}{v} y_v~\big\vert \mathscr{F}_k \big ].
\end{eqnarray}
Terms for which $\pi(k) = \tau(k)$ again cancel, and so we have
disregarded these terms from the right hand side.  It is also for this reason that the small correction appears in front of $X_k.$  From here it is possible to immediately deduce a sufficient \emph{a priori} bound on $X_k,$ as each term in this expectation is at most $\tfrac{\sqrt{d}}{n},$ so that
\[
|X_k| \leq 4\tfrac{\sqrt d}{n}.
\]
The conditional variance $\expect \left[X_k^2 ~\big\vert~
  \mathscr{F}_{k-1}\right]$ is not much more complicated.  Effectively, we take $\pi(k)$ to be uniformly distributed over $[n] \setminus \pi[k-1]$ and bound $\expect \left[X_k^2 ~\big\vert~ \mathscr{F}_{k-1}\right]$ by
\[
\expect \left[X_k^2 ~\big\vert~ \mathscr{F}_{k-1}\right]  \leq 4\expect \left[
x_k^2(\oneL{k}{v}y_v)^2 
+x_k^2(\oneL{k}{\pi(k)}y_{\pi(k)})^2 
+x_u^2(\oneL{u}{\pi(k)}y_{\pi(k)})^2 
+x_u^2(\oneL{u}{v}y_v)^2 
~\big\vert~ \mathscr{F}_{k-1} \right].
\] 
As we have ordered the $x_i,$ $x_u^2 \leq x_k^2.$  Further, by bounding all the $\oneL{a}{b}$ terms by $1,$ and using that $v$ is marginally distributed as $\operatorname{Unif} \left\{ [n] \setminus \pi[k-1] \right\},$ this bound becomes
\[
\expect \left[X_k^2 ~\big\vert~ \mathscr{F}_{k-1}\right]  \leq 16\expect \left[
x_k^2y_v^2 
~\big\vert~ \mathscr{F}_{k-1} \right].
\]
Upon explicit calculation, we see that
\[
\expect \left[
y_v^2 
~\big\vert~ \mathscr{F}_{k-1} \right]
= \frac{1}{n - k} \sum_{ [n] \setminus \pi[k-1] } y_v^2 \leq \frac{1}{n-k},
\]
where it has been used that $\|y\| \leq 1.$ Combining the above with
\eqref{starstar}, we see that
\begin{equation}
\label{truncation_need}
\expect \left[X_k^2 ~\big\vert~ \mathscr{F}_{k-1}\right]
\leq \left[\frac{n-k}{n-k+1}\right]^2\frac{16x_k^2}{n-k} \leq \frac{32x_k^2}{n}
\end{equation}
where it has been used that $k \leq r \leq \lceil \tfrac{n}{2}\rceil.$  Summing over all martingale increments,
\[
\sum_{i=1}^d \sum_{k=1}^{r} \frac{32x_k^2}{n} \leq \frac{32d}{n}.
\]
Thus the Freedman martingale bound becomes
\begin{align*}
\hspace{1in}&\hspace{-1in}\prob \left[
\left|\sum_{ \tilde{\mathcal{L}} } x_uA_{uv}y_v - \expect x_uA_{uv}y_v\right| > t\sqrt{d} 
\right] \leq 2\exp\left( \frac{ -nt^2}{64 + 8/3t}\right).
\end{align*}
\end{proof}

Let $\mathcal{L}_{\text{left}}$ be the set of vertices that appear in the first coordinate of some light couple, and choose $U \subseteq \mathcal{L}_{\text{left}}$ arbitrarily so that $\left|U\right| = \lceil {\left|  \mathcal{L}_{\text{left}}\right|}/{2} \rceil.$
It follows then that, if $U_1 := U$, and $U_2 :=  \mathcal{L}_{\text{left}}
  \setminus U_1$, 
\begin{align*}
&\prob\left[ \left|\sum_{ (u,v) \in \mathcal{L}} x_uA_{uv}y_v - \expec x_u A_{uv} y_v \right| > t\sqrt{d} \right]  \\ 
&~~~~\leq 
2\prob\left[ \max_{i=1,2} \left|\sum_{ (u,v) \in \mathcal{L}  \cap U_i \times [n] } x_uA_{uv}y_v - \expec x_u A_{uv} y_v \right| > \frac{t}{2}\sqrt{d} \right].
\end{align*}

From this point, it is possible to estimate
\[
\prob \left[ \exists~x,y \in \mathcal{T}~:~ \left|\sum_{\mathcal{L}} x_u M_{uv} y_v\right| > 2(2t+1)\sqrt{d}\right]
\]
by 
\[
\prob \left[ \exists~x,y \in \mathcal{T}~:~ \left|\sum_{\mathcal{L} \cap U \times [n]} x_u[A_{uv} -\expect A_{uv}]y_v\right| > t\sqrt{d}\right]
\]
Applying the union bound and Theorem~\ref{thm:lightlargedeviation}, we
see now that 
\[
\prob \left[ \exists~x,y \in \mathcal{T}~:~ \left|\sum_{\mathcal{L}}
    x_u M_{uv} y_v\right| > 2(2t+1)\sqrt{d}\right] \leq 
C \left[\frac{(2+\delta)\sqrt{2e \pi}}{2\delta} \right]^{2n} \exp\left(\frac{-nt^2}{64 + 8t/3} \right),
\]
so that taking $e - 2 \geq \delta \geq \tfrac 12$ and $t = 27,$ it is seen that this probability decays exponentially fast, and we have proven 
\begin{thm}
\label{thm:lightcouplecontribution}
There are universal constants $C$ and $K$ sufficiently large and $c > 0$ so that for $ e - 2 \geq \delta \geq \tfrac 12$ and except for with probability at most 
\[
K\exp\left( -cn \right),
\]
there is no pair of vectors $x,y \in \mathcal{T}$ having
\[
\left| \sum_{(u,v) \in \mathcal{L}} x_u M_{uv}y_v \right| \geq C\sqrt{2d}.
\]
It is possible to take $C =110.$ 
\end{thm}

\subsection{Controlling the contribution of the heavy couples.} 
\begin{lemma}[Discrepancy]
\label{lem:regularity}
For any two vertex sets $A$ and $B$, let $e(A,B)$ denote the number of
directed edges from $A$ to $B$ that result as a form $\pi_i(a) = b$
for some $1 \leq i \leq d,$ $a \in A$ and $b \in B.$  Let $\mu(A,B) =
|A||B|\tfrac{d}{n}.$  For every $m>0,$ there are constants $c_1 \geq e$ and
$c_2$ so that for every pair of vertex sets $A$ and $B$, except with
probability $n^{-m}$, exactly one of the following properties holds
\begin{enumerate}
\item either 
$
\tfrac{e(A,B)}{\mu(A,B)} \leq c_1~,
$
\item or 
$
e(A,B)\log \tfrac{e(A,B)}{\mu(A,B)} \leq c_2( |A| \vee |B|)\log \tfrac{n}{|A| \vee |B|}
$
\end{enumerate} 
 It is possible to take $c_1 =e^4$ and $c_2 = 2e^2(6+m).$ 
\end{lemma}

To prove this lemma, we rely on a standard type of large deviation
inequality shown below, which mirrors the large deviation inequalities available for sums of i.i.d. indicators.
\begin{lemma}
\label{lem:edgelargedeviations}
For any $k \geq e,$ 
\[
\prob \left[ e(A,B) \geq k \mu(A,B) \right] \leq \exp( - k[ \log k -2] \mu).
\]
\end{lemma}

\begin{proof}
Let $e_\pi(A,B)$ denote the number $a \in A$ so that $\pi(a) \in B.$  It is possible to bound
\[
\prob \left[ e_\pi(A,B) = t \right] \leq \frac{[a]_t[b]_t}{t![n]_t},
\]
where we recall that $[a]_t = a (a-1) \ldots (a-t+1)$ is the falling factorial or Pochhammer symbol.
Using the fact that  $[n]_t \geq e^{-t}n^t,$ this may be bounded as
\[
\prob \left[ e_\pi(A,B) = t \right] \leq \frac{a^tb^te^{t}}{t!n^t},
\]
so that the Laplace transform of $e_\pi(A,B)$ can be estimated as
\[
\expect \left[ \exp( \lambda e_\pi(A,B)) \right]
\leq \sum_{t=0}^\infty e^{\lambda t} \frac{a^tb^te^{t}}{t!n^t} 
= \exp\left[ 
\frac{abe^{1+\lambda}}{n}.
\right]
\]
Thus by Markov's inequality, we have
\begin{align*}
\prob \left[ e(A,B) \geq k \mu(A,B) \right]
&\leq \frac{\expect \left[ \exp\left(\lambda \sum_{i=1}^d e_{\sigma_i}(A,B) \right)\right]}{e^{-k\lambda \mu}} \\
&\leq \exp\left[ \mu e^{1+\lambda} - k\lambda\mu \right],
\end{align*}
where $\lambda >0$ is any positive number and $\mu = \mu(A,B).$  Taking $1+\lambda = \log k,$ valid for $k > e,$ it follows that 
\[
\prob \left[ e(A,B) \geq k \mu(A,B) \right]
\leq \exp\left[ -k(\log k - 2)\mu \right],
\] 
for $k \geq e.$
\end{proof}   

Armed with Lemma \ref{lem:edgelargedeviations}, we can proceed with the
proof of Lemma \ref{lem:regularity}. 

\begin{proof}[Proof of Lemma~\ref{lem:regularity}]
If either of $|A|$ or $|B|$ is greater than $\tfrac{n}{e},$ then $e(A,B) \leq (|A| \vee |B|)d,$ so that
\[
\frac{e(A,B)}{\mu(A,B)} \leq \frac{nd(|A| \vee |B|)}{|A||B|d} = \frac{n}{|A| \wedge |B|} \leq e.
\]
Thus, it suffices to deal with the case that both $A$ and $B$ are less
than $\tfrac{n}{e}.$  In what follows, we will think of $a$ and $b$ as
being the sizes of $|A|$ and $|B|$ in preparation to use a union
bound.  Let $k = k(a,b,n)$ be defined as $k = \max \{k^{*},
\frac{1}{e} \}$, where $k^{*}$ satisfies 
\[
 k^*\log k^*  = \frac{(6+m)(a \vee b)n}{abd} \log \frac{n}{a \vee b},
\]
or $\tfrac{1}{e},$ whichever is larger.  When $a \vee b \leq \tfrac{n}{e},$ it follows that
\begin{align*}
(6 + m)(a \vee b) \log\tfrac{n}{a \vee b}
& \geq 
2a\log\tfrac{n}{a} 
+2b\log\tfrac{n}{b} 
+(2 + m)(a \vee b) \log\tfrac{n}{a \vee b}, \\
\intertext{where we have used the monotonicity of $x\log \tfrac{n}{x}$
  on $[1,\tfrac{n}{e}]$; thus}
(6 + m)(a \vee b) \log\tfrac{n}{a \vee b} & \geq 
a(1 + \log\tfrac{n}{a})
+b(1 + \log\tfrac{n}{b})
+(2 + m)\log n.
\end{align*} 
Exponentiating,
\[
\exp\left[ 
k \log k \tfrac{abd}{n}
\right]
\geq \left(\tfrac{ea}{n}\right)^n
\left(\tfrac{eb}{n}\right)^n n^{2+m},
\]
if $k \geq \tfrac{1}{e}.$ 
It follows that
\begin{align*}
&\prob \left[ 
\exists A,B~\text{with}~|A|=a,~|B|=b,~\text{so that}~e(A,B) \geq e^2k(a,b)\mu(A,B) 
\right] \\
&\hspace{2in}\leq {n \choose a}{n \choose b}\exp( -e^2k[ \log k] \mu) \leq n^{-2-m}.
\end{align*}
Moreover, applying this bound to all $a$ and $b,$ it follows that
\[
e(A,B) \leq e^2 k(|A|,|B|) \mu(A,B),
\]
except with probability smaller than $n^{-m}.$
If for two sets $A$ and $B,$ $k=\tfrac{1}{e},$ then 
\[
e(A,B) \leq e \mu(A,B),
\]
and we are in the first case of the discrepancy property, for $c_1 \geq e.$  Otherwise,
\[
e(A,B) \log k \leq e^2 k\log k \mu(A,B) = e^2(6+m)(a \vee b) \log \frac{n}{a \vee b},
\]
and noting that $k \geq \frac{e(A,B)}{e^2\mu(A,B)},$ it follows that 
\[
\tfrac12 e(A,B) \log \frac{e(A,B)}{\mu(A,B)} \leq e(A,B) \log \frac{e(A,B)}{e^2 \mu(A,B)}
\leq e^2(6+m)(a \vee b) \log \frac{n}{a \vee b},
\]
when $\frac{e(A,B)}{\mu(A,B)} \geq e^4.$  If this is not the case, then we are again in the first case of the discrepancy property, taking $c_1 \geq e^4.$  Taking $c_1 = e^4,$ it follows that we may take $c_2 = 2e^2(6+m).$  
\end{proof}

The discrepancy property implies that there are no dense subgraphs, and thus the contribution of the heavy couples is not too large.

\begin{lemma}
\label{lem:discrepancy}
If the discrepancy property holds, with associated constants $c_1$ and $c_2$, then 
\[
 \sum_{ \{u,v\} \in \mathcal{H}} \left|x_uA_{u,v}y_v\right|  \leq C\sqrt{d},
\]
for some constant $C$ depending on $c_1,c_2,$ and $\delta.$
\end{lemma}

\begin{proof}
The method of proof here is essentially identical to Kahn and Szemer\'edi or Feige and Ofek (see~\cite{FKSz} or~\cite{FeigeOfek}).  We provide a proof of this lemma for completeness as well as to establish the constants involved.  We will partition the summands into blocks where each term $x_u$ or $y_v$ has approximately the same magnitude.  Thus let $\gamma_i = 2^i \delta,$ and put
\begin{align*}
A_i &= \left\{ u~\big\vert~ \tfrac{\gamma_{i-1}}{\sqrt n} \leq |x_u| < \tfrac{\gamma_{i}}{\sqrt n} \right\},& & 1 \leq i \leq \log \lceil \sqrt{n} \rceil. \\
B_i &= \left\{ u~ \big\vert~ \tfrac{\gamma_{i-1}}{\sqrt n} \leq |y_u| < \tfrac{\gamma_{i}}{\sqrt n} \right\},& & 1 \leq i \leq \log \lceil \sqrt{n} \rceil.
\end{align*}
Let $\hat {\mathcal{H}}$ denote those pairs $(i,j)$ so that $\gamma_i\gamma_j \geq \sqrt{d}.$  The contribution of the absolute sum can, in these terms, be bounded by
\[
 \sum_{ (u,v) \in \mathcal{H}} \left|x_uM_{u,v}y_v\right|  \leq 
 \sum_{ (i,j) \in \hat{\mathcal{H}}} \frac{\gamma_i\gamma_j}{n} e(A_i,B_j). 
\]
Let $\lambda_{i,j} = \tfrac{e(A_i,B_j)}{\mu(A_i,B_j)}$ denote the discrepancy, which can be controlled using Lemma~\ref{lem:regularity}.  In terms of this quantity, the bound becomes
\[
 \sum_{ (u,v) \in \mathcal{H}} \left|x_uM_{u,v}y_v\right|  \leq
   \sum_{ (i,j) \in \hat{\mathcal{H}}} \frac{\gamma_i\gamma_j}{n} \lambda_{i,j} |A_i||B_j|\tfrac{d}{n}. 
\]
In this form, the magnitudes of each of the quantities are somewhat
opaque.  Consider the sum $\sum_{i} |A_i| \frac{\gamma_i^2}{n};$ it is
at most $4\|x\|^2= 4.$  In particular, it is of constant order.  Thus let $\alpha_i =|A_i| \frac{\gamma_i^2}{n}$ and $\beta_j = |B_j| \frac{\gamma_j^2}{n}.$  This allows the bound to be rewritten as
\[
d \sum_{ (i,j) \in \hat{\mathcal{H}}} \frac{\gamma_i^2|A_i|}{n}\frac{\gamma_j^2|B_j|}{n} \frac{\lambda_{i,j}}{\gamma_i\gamma_j} = 
\tfrac{d}{\sqrt{d}} \sum_{ (i,j) \in \hat{\mathcal{H}}} \alpha_i \beta_j \frac{\lambda_{i,j}\sqrt{d}}{\gamma_i\gamma_j}.
\]
This exposes the quantity $\sigma_{i,j} =
\frac{\lambda_{i,j}\sqrt{d}}{\gamma_i\gamma_j}$ as having some special
importance.  In effect, we will show that either for fixed $i,$
$\sum_{j} \sigma_{i,j} \beta_j$ has constant order,  or for fixed $j$, $\sum_{i} \sigma_{i,j}\alpha_i$ has constant order.

In what follows, we will bound the contribution of the summands where $|A_i| \geq |B_j|.$  By symmetry, the contribution of the other summands will have the same bound.  The heavy couples will now be partitioned into $6$ classes $\{ \hat{\mathcal{H}}_i\}_{i=1}^6$ where their contribution is bounded in a different way.  Let $\hat{\mathcal{H}}_i \subseteq \hat{\mathcal{H}}$ be those pairs $(i,j)$ which satisfy the $i^{th}$ property from the following list but none of the prior properties:
\begin{enumerate}
\item $\sigma_{i,j} \leq c_1.$
\item $\lambda_{i,j} \leq c_1.$
\item $\gamma_j > \tfrac14 \sqrt{d}\gamma_i.$
\item $\log \lambda_{i,j} > \tfrac{1}{4}\left[2\log \gamma_i + \log \tfrac{1}{\alpha_i} \right].$
\item $2\log \gamma_i \geq \log\tfrac{1}{\alpha_i}.$
\item $2\log \gamma_i < \log\tfrac{1}{\alpha_i}.$
\end{enumerate}
The last properties are better understood when the second case of the discrepancy property is expressed in present notation.  In its original form, it states
\[
e(A_i,B_j) \log \lambda_{i,j} \leq c_2 |A_i| \log \tfrac{n}{|A_i|}.
\]
Substituting $\gamma_i^2/\alpha_i$ for $n/|A_i|$ and multiplying both sides of this equation through by $\frac{\gamma_i }{|B_j|\gamma_j \sqrt{d} \log \lambda_{i,j}}$ produces the equivalent form
\[
\sigma_{i,j}\beta_j \leq c_2 \frac{\gamma_j}{\sqrt{d}\gamma_i}\frac{\left[2\log \gamma_i + \log \tfrac{1}{\alpha_i} \right]}{\log \lambda_{i,j}}.
\]
Thus, the last $3$ cases cover each of the possible dominant $\log$ terms in this bound.
\subsubsection{ Bounding the contribution of $\hat{\mathcal{H}}_1$ and $\hat{\mathcal{H}}_2.$} 
In either of these situations, we have a bound on $\sigma_{i,j}.$  Especially, either $\sigma_{i,j} \leq c_1$ or, all the discrepancies $\lambda_{i,j}$ are uniformly bounded by $c_1.$  As 
\[
\sigma_{i,j} = \frac{\lambda_{i,j}\sqrt{d}}{\gamma_i \gamma_j},
\]
and $\gamma_i \gamma_j \geq \sqrt{d},$ 
\[
\sigma_{i,j} \leq c_1
\]
for both cases.
\subsubsection{ Bounding the contribution of $\hat{\mathcal{H}}_3$.}
\emph{For these terms, we fix $j.$}  
In this case, the magnitudes of the entries corresponding to $j$  of
$y_v$ dominate those of the entries corresponding to $i$  of $x_u.$  However, by regularity $e(A_i,B_j) \leq |B_j| d,$ so that the discrepancy $\lambda_{i,j}$ is at most $\frac{ n }{|A_i|} = \frac{ \gamma_i^2}{\alpha_i}.$  
\[
\sum_{ i~:~(i,j) \in \hat{\mathcal{H}_3}} \alpha_i \sigma_{i,j}
= \sum_{ i~:~(i,j) \in \hat{\mathcal{H}_3}} \alpha_i \frac{\lambda_{i,j}\sqrt{d}}{\gamma_i\gamma_j} 
\leq \sum_{ i~:~(i,j) \in \hat{\mathcal{H}_3}} \frac{\gamma_i\sqrt{d}}{\gamma_j} \leq 8, 
\]
where in the last step it has been used that the sum is geometric with leading term less than $4\gamma_j /\sqrt{d}.$  

\subsubsection{ Bounding the contribution of $\hat{\mathcal{H}}_4$.}
\emph{For these terms, we fix $i.$}  
We are not in case $(2),$ and it follows that the second case of the discrepancy property holds.  In present notation
\[
\sigma_{i,j}\beta_j \leq c_2 \frac{\sqrt{d}\gamma_j}{d\gamma_i}\frac{\left[2\log \gamma_i + \log \tfrac{1}{\alpha_i} \right]}{\log \lambda_{i,j}}
\leq \frac{4c_2\gamma_j}{\gamma_i \sqrt{d}},
\]
where the hypothesis has been used.  As we are not in case $(3)$, the sum of these terms is bounded as
\[
\sum_{ j~:~(i,j) \in \hat{\mathcal{H}_4}} \beta_j \sigma_{i,j}
\leq 2c_2,
\]
where it has been used that the sum above has a geometric dominator with leading term at most $\tfrac 14\gamma_i\sqrt{d}.$ 

\subsubsection{ Bounding the contribution of $\hat{\mathcal{H}}_5$.}
\emph{For these terms, we fix $i.$}  
Again, the second case of the discrepancy property holds.  Now, in addition,
\[
\log \lambda_{i,j} \leq \tfrac{1}{4}\left[2\log \gamma_i + \log \tfrac{1}{\alpha_i} \right] \leq \log \gamma_i,
\]
i.e. that $\lambda_{i,j} \leq \gamma_i.$  Furthermore, we are not in case $(1)$ so $ c_1 \leq \sigma_{i,j} = \tfrac{\lambda_{i,j}\sqrt{d}}{\gamma_i \gamma_j} \leq \frac{\sqrt{d}}{\gamma_j}.$  Thus the second discrepancy bound becomes
\[
\sigma_{i,j}\beta_j \leq c_2 \frac{\sqrt{d}\gamma_j}{d\gamma_i}\frac{\left[2\log \gamma_i + \log \tfrac{1}{\alpha_i} \right]}{\log \lambda_{i,j}} \leq
c_2 \frac{\gamma_j 4\log \gamma_i}{\sqrt{d}\gamma_i\log c_1}
\leq \frac{4c_2}{c_1} \frac{\gamma_j}{\sqrt{d}},
\] 
where it has been used that $\gamma_i \geq \lambda_{i,j} \geq c_1 \geq
e$, and that $\log x / x$ is monotonically decreasing for $x > e.$
Thus, 
\[
\sum_{ j~:~(i,j) \in \hat{\mathcal{H}_5}} \beta_j \sigma_{i,j}
\leq\sum_{ j~:~(i,j) \in \hat{\mathcal{H}_5}} \frac{4c_2}{c_1} \frac{\gamma_j}{\sqrt{d}} 
\leq \frac{8c_2}{c_1^2},
\]
where it has been used that the second sum above is geometric with largest term $\sqrt{d}/c_1.$  
\subsubsection{ Bounding the contribution of $\hat{\mathcal{H}}_6$.}
\emph{For these terms, we fix $j.$}  
The second case of the discrepancy property holds and in addition,
\[
\log \lambda_{i,j} \leq \tfrac{1}{4}\left[2\log \gamma_i + \log \tfrac{1}{\alpha_i} \right] \leq \tfrac12 \log \tfrac{1}{\alpha_i}.
\]
This implies that $\sigma$ satisfies the asymmetric bound $\sigma_{i,j} \leq \tfrac{1}{\alpha_i}\tfrac{\sqrt{d}}{\gamma_i\gamma_j}.$  Thus,
\[
\sum_{ i~:~(i,j) \in \hat{\mathcal{H}_6}} \alpha_i \sigma_{i,j}
\leq
\sum_{ i~:~(i,j) \in \hat{\mathcal{H}_6}}\tfrac{\sqrt{d}}{\gamma_i\gamma_j} \leq 2, 
\] 
where it has been used that the sum above is geometric with leading term $\tfrac{1}{\sqrt{d}}$ (which follows as $\gamma_i\gamma_j \geq \sqrt{d}$). 

\subsubsection{ Assembling the bound }
We must sum the contributions of each of the classes of couples.   Recall that we must double the contribution here because we have only considered couples where $|A_i| \geq |B_j|.$  In each of the cases outlined above, it only remains to sum over the $\alpha_i$ or $\beta_j$ in each bound.  Doing so contributes a factor of $4$ to each bound, so that the constant can be given by 
\[
2\left[
16c_1+32+8c_2+
\frac{32c_2}{c_1^2}
+8
\right]
\]
\end{proof}

\subsection{Finalizing the proof of Theorem~\ref{thm:eigbound} }
\begin{proof}
We will take $\delta = \tfrac 12.$  With $m$ given, it follows the
discrepancy property (Lemma~\ref{lem:regularity}) holds with
probability at least  $1 - n^{-m},$ and with constants $c_1 = e^4$ and
$c_2 = 2e^2(6+m).$  Therefore, by Lemma~\ref{lem:discrepancy}, for any
two $x, y \in \mathcal{T},$ the contribution of the heavy couples to
$x^tMy$ (which is at most twice the contribution of $x^tAy$, given
that the bounds hold for \emph{all} $x$ and $y$) is at most
\[
4\left[
16c_1+32+8c_2+
\frac{32c_2}{c_1^2}
+8
\right] \sqrt{d}
\leq
(8854+ 585m) \sqrt{d}.
\]   
By Theorem~\ref{thm:lightcouplecontribution}, with probability at
least $(1- C\exp(-cn)$ for some universal constants $C > 0$ and $c >0,$ the contribution of the
light couples is never more than $110 \sqrt{d}$. Hence
\[
\sup_{x,y \in \mathcal{T} } |x^t M y| \leq (8964 + 585m) \sqrt{d},
\]
 except with probability at most $n^{-m} + C\exp(-cn).$  At last, this
 implies that $\lambda_2 \vee |\lambda_n| \leq 4(8964 + 585m) \sqrt{d}$, except
 with probability at most $n^{-m} + C\exp(-cn).$
\end{proof}


\section{Linear statistics of eigenvalues}\label{sec:linear}  

We now connect 
Section~\ref{sec:nbw} to linear eigenvalue statistics
of the adjacency matrix of $G_n$.
Let $\{T_n(x)\}_{n \in \mathbb{N}}$ be the Chebyshev polynomials of the first kind 
on the interval $[-1,1]$.
We define a set of polynomials
\begin{align}\label{eq:whatisgamma}
  \Gamma_0(x) & = 1~, \\
  \Gamma_{2k}(x) &= 2 T_{2k}\left (\frac{x}{2} \right ) + \frac{2d-2}{(2d-1)^k}~,~~\forall ~k \geq 1~, \\
  \Gamma_{2k+1}(x) &= 2 T_{2k+1}\left(\frac{x}{2}\right)~, ~~\forall ~k \geq 0~.
\end{align}
We note that much of the following proposition
can be found in Lemma~10.4 of \cite{friedmanalon}.
\begin{prop}\label{prop:eigenvaluewalks}
  Let $A_n$ be the adjacency matrix of $G_n$, and
  let $\lambda_1\ge\cdots\ge\lambda_n$ be the eigenvalues
  of $(2d-1)^{-1/2}A_n$.  Then
  \begin{align*}
    \N[n]{k}:=\sum_{i=1}^n\Gamma_k(\lambda_i)&=(2d-1)^{-k/2}\CNBW[n]{k}.
  \end{align*}
\end{prop}
\begin{proof}
To show the above, we will first use the Chebyshev polynomials of the second kind on $[-1,1]$, namely, $\{U_n\}_{n \in \mathbb{N}}$. 

  Let 
  \begin{align}\label{eq:orthogKM}
    p_k(x)=U_k\left(\frac{x}{2}\right)
     - \frac{1}{2d-1}U_{k-2}\left(\frac{x}{2}\right).
  \end{align}
  It is known  \cite[eqn.~12]{ABLS} that
    $(2d-1)^{-k/2}\NBW[n]{k}=\sum_{i=1}^np_k(\lambda_i)$.
    We thus proceed by relating $\CNBW[n]{k}$ to $\NBW[n]{k}$.
  
  A closed non-backtracking walk of length $k$ is either cyclically
  non-backtracking or can be obtained from a closed non-backtracking
  walk of length $k-2$ by ``adding a tail,'' i.e., adding a new step
  to the beginning of the walk and its reverse to the end.
  For any closed cyclically non-backtracking walk of length
  $k-2$, we can add a tail  
   in $2d-2$ ways.
  For any closed non-backtracking walk of length $k-2$ that is not cyclically
  non-backtracking, we can add a tail in $2d-1$ ways.  
  Hence for $k\geq 3$,
  \begin{align*}
    \NBW[n]{k}&=\CNBW[n]{k}+(2d-2)\CNBW[n]{k-2}+(2d-1)
      \left(\NBW[n]{k-2}-\CNBW[n]{k-2}\right)\\
      &= \CNBW[n]{k}+(2d-1)\NBW[n]{k-2}-\CNBW[n]{k-2}.
  \end{align*}
  Applying this relation iteratively
  and noting that $\CNBW[n]{k}=\NBW[n]{k}$ for $k=1,2$, we have
  \begin{align*}
    \CNBW[n]{k} =
      \NBW[n]{k} - (2d-2)
      \left(\NBW[n]{k-2}+\NBW[n]{k-4}+
      \cdots+\NBW[n]{a}\right)
  \end{align*}
  with $a=2$ if $k$ is even and $a=1$ if $k$ is odd.
 Observe now that
  \begin{align*}
    \Gamma_{2k}(x) &= p_{2k}(x)-(2d-2)\left(\frac{p_{2k-2}(x)}{2d-1}
     +\frac{p_{2k-4}(x)}{(2d-1)^2}+\cdots+\frac{p_{2}(x)}{(2d-1)^{k-1}}\right),\\
     \intertext{and}
    \Gamma_{2k-1}(x) &= p_{2k-1}(x)-(2d-2)\left(\frac{p_{2k-3}(x)}{2d-1}
     +\frac{p_{2k-5}(x)}{(2d-1)^2}+\cdots+
     \frac{p_{1}(x)}{(2d-1)^{k-1}}\right)~.
  \end{align*}
A quick calculation shows now that 
\begin{eqnarray*}
\Gamma_{2k}(x) & = & U_{2k} \left( \frac{x}{2} \right ) - U_{2k-2} \left (\frac{x}{2} \right) + \frac{2d-2}{(2d-1)^{k/2}}~, ~~\mbox{while}\\
\Gamma_{2k+1}(x) & = & U_{2k+1} \left ( \frac{x}{2} \right ) - U_{2k-1} \left ( \frac{x}{2} \right )~,
\end{eqnarray*}
and the rest follows from the fact that $T_{k}(x) = \frac{1}{2} \left (U_{k}(x)- U_{k-2}(x) \right)$. 
\end{proof}
\bigskip  
  
The weak convergence of the sequence $(\CNBW[n]{k},\; 1\le k \le r_n)$ in Theorem \ref{thm:dfixedtight} allows us to establish limiting laws for a general class of linear functions of eigenvalues. First we will make some canonical choices of parameters $\{r_n\}$. Define
\eq\label{eq:choosern}
r_n = \frac{\beta\log n}{\log(2d-1)}, \quad \text{for some $\beta < 1/2$.}
\en
Note that $2r_n \log (2d-1) = 2\beta \log n$, which shows \eqref{eq:2d}, even when $d$ grows with $n$. 

We now need another definition. Let $h$ be a function on $\RR$ such that 
\eq\label{eq:whatish}
h(r_n) \ge \log (2d-1), \quad \text{for all large enough $n$.}
\en
This definition is not so important when $d$ is fixed, since a
constant $h(x)\equiv\log(2d-1)$ for all $x\in \RR$ is a good choice.
However, when $d$ grows with $n$, an appropriate choice needs to be made. For example when $2d-1= (\log
n)^\gamma$ for some $\gamma >0$,  one may take 
\eq\label{eq:computehk}
h(x) = C\log x,\quad \text{for some large enough positive constant $C$}.   
\en


For our next result, we will use some theorems from Approximation
Theory. Recall that every function $f$ on $[-1, 1]$ which is square-integrable with respect to
the 
arc-sine law has a series expansion with respect to the
Chebyshev polynomials of the first kind. Good references for
approximation theory and the Chebyshev polynomials are the book
\cite{MH} and the (yet unpublished) book \cite{ATAP11}. 

Recall the polynomials $\Gamma_k(x)$ as defined in \eqref{eq:whatisgamma}; if a function
has a series expansion in terms of Chebyshev polynomials of the first
kind, $T_k(x)$, on $[-1,1]$, then it has a series expansion in terms
of $\Gamma_k(x)$ on $[-2, 2]$. 

We recall the definition of a Bernstein ellipse of radius $\rho$.

\begin{defn}
Let $\rho>1$, and let $\mathcal{E}_B(\rho)$ be the image of the circle of radius $\rho$,
centered at the origin, under the map $f(z) = \frac{z+ z^{-1}}{2}$. We
call $\mathcal{E}_B(\rho)$ the Bernstein ellipse of radius $\rho$. The ellipse has foci at $\pm 1$, and the sum
of the major semiaxis and the minor semiaxis is exactly
$\rho$. 
\end{defn}


To prove our main  result for $d$ fixed, we first need a lemma. 

\begin{lemma} \label{lem:prelim_dfixed}
Suppose that $d\geq 2$ is fixed. Let $f$ be a function defined on
$\mathbb{C}$ which is analytic inside a Bernstein ellipse of radius
$2\rho$, where $\rho = (2d-1)^{\alpha}$, for some $\alpha>2$, and
such that $|f(z)| < M$ inside this ellipse. 

Let $f(x) = \sum_{i=0}^{\infty} c_i \Gamma_i(x)$ for $x$ on $[-2,2]$
(the existence, as well as uniform convergence of the series on $[-2,
2]$,  is guaranteed  by the fact that $f$ is analytic on $[-2,2]$). 

Then the following things are true:
\begin{itemize}
\item[(i)] The expansion of $f(x)$ in terms of $\Gamma_i(x)$ actually
  converges uniformly on $[-2-\epsilon, 2+ \epsilon]$ for some small enough
$\epsilon>0$. 
\item[(ii)] The aforementioned series expansion also converges pointwise on
  $[2, \frac{2d}{\sqrt{2d-1}}]$. 
\item[(iii)] If $f_k :=\sum_{i=0}^k c_i \Gamma_i$ is the $k$th truncation
  of this (modified) Chebyshev series for $f$, then, for a small enough $\epsilon$, 
\[
\sup_{0\le\abs{ x} \le 2+\epsilon} \abs{f(x) - f_k(x)} \le M' \left(2d-1\right)^{-\alpha' k}~,
\]
where $2<\alpha'<\alpha$, and $M'$ is a constant independent of $k$. 
\item[(iv)] For all $k \in \mathbb{N}$, let $b_k = \frac{1}{(2d-1)^k}$,
  and let $\omega_k$ be the sequence of weights described in Theorem
  \ref{thm:dfixedtight}. Then the sequence of coefficients $\{c_k\}_{k
    \in \mathbb{N}}$ satisfies
\[
\left( \frac{c_k}{(2d-1)^{k/2} \omega_k} \right)_{k\in \NN} \in \ltwo(\weight)~.
\]
\end{itemize}
\end{lemma}

\begin{proof} We will prove the facts (i) through (iv) in succession. 

Facts (i) and (ii) will use a particular expression for $T_n(x)$ outside
$[-1,1]$, namely, 
\begin{eqnarray} \label{new_t}
T_n(x) &=& \frac{(x-\sqrt{x^2-1})^n + (x + \sqrt{x^2-1})^n}{2}~.
\end{eqnarray}

For Fact (i), it is easy to see that if $x$ is in $[-2-\epsilon, 2+
\epsilon]$, and particularly for $\epsilon$ small enough,
\[
|\Gamma_k(x)| \leq C (1+ 3 \sqrt{\epsilon})^k~,
\]
 where $C$ is some constant independent of $k$. 

By Theorem 8.1 in \cite{ATAP11}, which first appeared in Section~61 
of \cite{bernstein12c}, it follows that
\begin{eqnarray} \label{bound_on_cs}
|c_k| \leq M' (2d-1)^{-\alpha k}~,
\end{eqnarray}
for some constant $M'$ which may depend on $M$ and $d$, but not on
$k$. 

Note that $1+ 3\sqrt{\epsilon} < (2d-1)^{\alpha}$, for any $d\geq 2$,
$\alpha>2$, and $\epsilon$ small enough. 

Consequently, the series $\sum_{k=0}^{\infty} c_k \Gamma_k(x)$ is
absolutely convergent on $[-2-\epsilon, 2+ \epsilon]$, and hence the
expansion of $f$ into this modified Chebyshev series is valid (and
absolutely convergent) on $[-2-\epsilon, 2+ \epsilon]$. This proves
Fact (i). 

Similarly, we now look on the interval $[2, \frac{2d}{\sqrt{2d-1}}]$,
and note that on that interval the expression for $T_n(x/2)$ will be
bounded from above by 
\[
|T_n(x/2)| < \frac{1+ (2d-1)^{n/2}}{2}~;
\]
indeed, this happens because $x/2 - \sqrt{x^2/4-1}$ is decreasing (and
maximally $1$,  at $x=2$) while $x/2+\sqrt{x^2/4-1}$ is increasing
(and maximally $(2d-1)^{n/2}$, at $x=2d/\sqrt{2d-1}$).

From here it follows once again that 
\[
|\Gamma_n(x)| \leq 2 (2d-1)^{n/2}~,
\]
on $[2, \frac{2d}{\sqrt{2d-1}}]$, and thus the series $\sum_{k=0}^{\infty} c_k \Gamma_k(x)$ is
absolutely convergent on this interval as well. The equality with the function $f$ follows from analyticity. This proves Fact (ii). 

Fact (iii) is an immediate consequence of \eqref{bound_on_cs}, by taking
$\epsilon$ small enough relative to $d$ and $\alpha$. 

Fact (iv) follows easily from the definitions of $\omega_k$, $\Theta_k$
(given in  Theorem \ref{thm:dfixedtight}), and from \eqref{bound_on_cs}. 
\end{proof}

We can now present our main result for the case when $d$ is fixed. 

\begin{thm}\label{thm:dfixedlinear}
Assume the same conditions on $f$  and notations as in 
Lemma~\ref{lem:prelim_dfixed}. 
Then the random variable $\sum_{i=1}^n f(\lambda_i) - nc_0$ converges in law to the infinitely divisible random variable 
\[
Y_f:=\sum_{k=1}^\infty \frac{c_k}{(2d-1)^{k/2}} \CNBW[\infty]{k}~.
\]
\end{thm}

\begin{rmk} There is a good explanation of why we must subtract 
  $nc_0$ in the statement of the above theorem.  Consider the
  Kesten-McKay density, normalized to have
  support $[-2,2]$:
  \begin{align*}
    \rho_{2d}(x) &= \frac{2d(2d-1)\sqrt{4-x^2}}{2\pi(4d^2-(2d-1)x^2)}.
  \end{align*}
  It is proved in \cite{mckay81} that in the uniform model
  of random $d$-regular graph, the random variable
  $n^{-1}\sum_{i=1}^nf(\lambda_i)$ converges in probability to 
  $\int_{-2}^2 f(x)\rho_d(x)dx$.
  This also holds for the present model; one can prove it
  by applying the contiguity results of \cite{GJKW}, or by
  using the above theorem to compute that
  $\limn n^{-1}\sum_{i=1}^n \lambda_i^k$ 
  is the $k$th moment of the Kesten-McKay law.
  
   If $\sum_{i=1}^n f(\lambda_i)$ converges in distribution 
   (without subtracting the constant), then $n^{-1}\sum_{i=1}^n f(\lambda_i)$ 
   converges to zero in probability. 
   Thus such a function $f$ must be orthogonal to one in the $\mathbf{L}^2$ space of the Kesten-McKay law.
     It has been shown in \cite[Example 5.3]{SODIN} that the polynomials $(p_k)$, defined in \eqref{eq:orthogKM}, along with the constant polynomial $p_0\equiv 1$ constitute an orthogonal basis for the $\ltwo$ space. The polynomials $(\Gamma_k)$, being linear combinations of $(p_k,\; k\ge1)$, are therefore orthogonal to one in that $\mathbf{L}^2$ space. 
    Hence for any $f$ of Theorem~\ref{thm:dfixedlinear}, the function
   $f - c_0$ is orthogonal to the Kesten-McKay law. 
\end{rmk}

\begin{proof} 
Armed with the results of Lemma \ref{lem:prelim_dfixed}, the proof is simple.

We first claim that
\[
\Y[n]{f}:=\sum_{k=1}^{r_n} c_k \N[n]{k}= \sum_{k=1}^{r_n} \frac{c_k}{(2d-1)^{k/2} \omega_k} \CNBW[n]{k} \omega_k 
\]
converges in law to $Y_f$ as $n$ tends to infinity. This follows from Theorem \ref{thm:dfixedtight} and Lemma \ref{lem:contimap} once we show that the sequence
\[ 
\left( \frac{c_k}{(2d-1)^{k/2} \omega_k}  \right)_{k\in \NN} \in \ltwo(\weight).
\] 
This is precisely Fact (iv) from Lemma \ref{lem:prelim_dfixed}. 

The result will now follow from Slutsky's theorem once we show that, for any $\delta >0$,
\eq\label{eq:inprobyf}
\lim_{n\rightarrow \infty} \P\left( \abs{\sum_{i=1}^n f(\lambda_i) - nc_0- \Y[n]{f}} > \delta \right) = 0. 
\en

The proof of \eqref{eq:inprobyf} has two components. Choose the parameter $\beta$ in \eqref{eq:choosern} such that $\alpha \beta < 1$. This also implies $\beta < 1/2$. We start by noting 
\[
nc_0+ \Y[n]{f}= \sum_{i=1}^n \sum_{k=1}^{r_n} c_k \Gamma_k(\lambda_i)= f_{r_n}(\lambda_1) + \sum_{i=1}^{n-1} f_{r_n}(\lambda_i).
\]
Recall that the first eigenvalue of $A_n$ is exactly $2d$,
irrespective of $n$. Thus, once we scale $A_n$ by $\sqrt{2d-1}$, by
Fact (ii) from Lemma \ref{lem:prelim_dfixed}, $f_{r_n}\left(\frac{2d}{\sqrt{2d-1}}\right)$ converges as a
  deterministic sequence to $f\left ( \frac{2d}{\sqrt{2d-1}} \right)$. Choose a large enough $n_1$ such that
\[
\abs{f_{r_n}\left(\frac{2d}{\sqrt{2d-1}}\right)- f \left(\frac{2d}{\sqrt{2d-1}}\right)} < \delta/4, \quad \text{for all $n \ge n_1$}.
\]

On the other hand, if we define the event
\[
A_n := \left\{ \abs{\lambda_i}\le 2 + \epsilon,\; \text{for all $i>1$}\right\},  
\]
Theorem 1.1 in \cite{friedmanalon}, shows that $\P\left( A_n\right)
\ge 1- cn^{-\tau}$, for some positive constants $c$ and $\tau$. On
this event, Fact (i) from Lemma \ref{lem:prelim_dfixed}, together with
\eqref{eq:choosern}
, implies that
\[
\sum_{i=2}^{n-1}\abs{f(\lambda_i) - f_{r_n}(\lambda_i)} \le (n-1) M \exp\left( -\alpha r_n \log(2d-1) \right)= M n\exp(-\alpha\beta \log n)= Mn^{-\alpha\beta+1}=o(1). 
\]
Choose a large enough $n_2$ such that the above number is less than $\delta/4$. 

Thus, for all $n \ge \max(n_1, n_2)$, we have
\[
\P\left( \abs{\sum_{i=1}^n f(\lambda_i) - nc_0- \Y[n]{f}} > \delta \right) \le P(A_n^c) =  cn^{-\tau}=o(1).
\]
This completes the proof. 
\end{proof}
\begin{rmk}
  We now take a moment to demonstrate how to compute the limiting distribution
  of $\sum_{j=1}^n\Gamma_k(\lambda_j)$ when $d=1$ using the results
  of \cite{BAD}, and we show that it is consistent with our own results.
  (Though in this paper we focus on $d \geq 2$, our techniques apply for $d=1$, too,
  and prove nearly the same result as Theorem~\ref{thm:dfixedlinear}.)
  Let $M_n$ be a uniform random $n\times n$ permutation
  matrix with eigenvalues $e^{2\pi i \varphi_1},\ldots,e^{2\pi i\varphi_n}$
  on the unit circle. Let $A_n=M_n+M_n^T$ with eigenvalues 
  $\lambda_1,\ldots,\lambda_n$, which satisfy 
  $\lambda_j=2\cos(2\pi\varphi_j)$. 
  We define $f(x)=\Gamma_k(2\cos(2\pi x))=2\cos(2\pi k x)+c_k$,
  where $c_k=0$ when $k$ is odd and $c_k=(2d-2)/(2d-1)^{k/2}$ when $k$
  is even.  Then $\sum_{j=1}^n\Gamma_k(\lambda_j)=\sum_{j=1}f(\varphi_j)$.
  
  Theorem~1.1 of \cite{BAD} gives the characteristic function of the
  limiting distribution $\mu_f$ of 
  $\sum_{j=1}f(\varphi_j)-\E \sum_{j=1}f(\varphi_j)$
  as
  \begin{align*}
    \hat{\mu}_f(t) &=\exp\left(\int(e^{itx}-1-itx)dM_f(x)\right)
  \end{align*}
  with $M_f$ given by
  \begin{align*}
    M_f &= \sum_{j=1}^{\infty}\frac{1}{j}\delta_{jR_j(f)},\\
    R_j(f) &= \frac{1}{j}\sum_{h=0}^{j-1}f\left(\frac{h}{j}\right)-\int_0^1
      f(x)dx.
  \end{align*}
  It is straightforward to calculate that
  \begin{align*}
    R_j(f)=\begin{cases}2&\text{if $j|k$,}\\0&\text{otherwise.}
    \end{cases}
  \end{align*}
  Thus we find
  \begin{align*}
    \hat{\mu}_f(t)&=\exp\left(\sum_{j|k}\frac{1}{j}(e^{2itj}-1)-2it)\right),
  \end{align*}
  which is the characteristic function of $\CNBW{k}-\E\big[\CNBW{k}\big]$
  for $d=1$ (note that $a(d,k)=2$ in this case).
\end{rmk}


Finally,  we consider now the case of growing degree $d = d_n $ and
the relationship between $d_n$ and $r_n$, as given in the statement of
Theorem \ref{thm:dgrowstight} and in \eqref{eq:choosern}. Although we have chosen not to use the
notation $d_n$ elsewhere in the paper, we will use it here, to
emphasize each pair $(d_n, r_n)$. For our results to be applicable, we
will need that both $d_n$ \emph{and} $r_n$ grow to $\infty$. 

We will first remove the dependence on $d_n$ for our orthogonal
polynomial basis, making them scaled Chebyshev.  Define
\begin{align}\label{eq:whatisphi}
  \Phi_0(x)&=1,\\
  \Phi_{k}(x) &= 2 T_k \left ( \frac{x}{2} \right )~, \quad k \ge 1.
\end{align}
If $A_n$ is the adjacency matrix of $G_n$ and
$\lambda_1\geq\cdots\geq\lambda_n$ are the eigenvalues
of $(2d_n-1)^{-1/2}A_n$ and $k \ge 1$, then
\begin{align*}
  \sum_{i=1}^n\Phi_k(\lambda_i) &= 
    \begin{cases}
      (2d_n-1)^{-k/2}\left(\CNBW[n]{k}
         - (2d_n-2)n\right) &\text{if $k$ is even,}\\
      (2d_n-1)^{-k/2}\CNBW[n]{k} & \text{if $k$ is odd.}
    \end{cases}
\end{align*}

Please note from \eqref{eq:whatisntilde} that
  \begin{align*}
    \Nc[n]{k} &= \begin{cases}
      \sum_{i=1}^n\Phi_k(\lambda_i)
      -(2d_n-1)^{-k/2}\big(\mu_k(d_n)-(2d_n-2)n\big)&\text{if $k$ is even}\\
      \sum_{i=1}^n\Phi_k(\lambda_i)
      -(2d_n-1)^{-k/2}\mu_k(d_n)&\text{if $k$ is odd}\end{cases}
  \end{align*}
  
  Our final result is very similar in spirit to 
  Theorem~ \ref{thm:dfixedlinear}, and we will need a helpful tool like
  Lemma~\ref{lem:prelim_dfixed} to make it work. 

\begin{lemma} \label{lem:prelim_dgrows}
Suppose now that $d_n$, $r_n$ are growing
    with $n$ and governed by \eqref{eq:choosern}. Consider the polynomials $\Phi_k$ as in
    \eqref{eq:whatisphi}. Let $f$ be an entire function on
    $\mathbb{C}$. Let $a>1$ be a fixed real number. Then 
\begin{itemize} \item[(i)] $f$ admits an absolutely convergent (modified) Chebyshev series
    expansion 
\[
f(x) = \sum_{i=0}^{\infty} c_i \Phi_i(x)
 \]
on $[-a, a]$;
\item[(ii)]  for some choice of weights $\weight=(b_k/k^2\log
k)_{k\in \NN}$ from Theorem \ref{thm:dgrowstight}, the sequence of
coefficients $(c_k)_{k\in \NN}$ satisfies 
 \begin{eqnarray} \label{coeff_growth}
\left( \frac{c_k}{\omega_k}  \right)_{k\in \NN}  &\in& \ltwo\left(
  \weight \right).
\end{eqnarray}
\end{itemize}

\end{lemma}

\begin{proof}
Both Facts (i) and (ii) follow in the same way as the proofs of Facts
(i) and (ii) from Lemma \ref{lem:prelim_dfixed}, noting that, since $f$ is entire,  it is sufficient to choose a
Bernstein ellipse of radius large enough. This will provide a fast-enough
decaying geometric bound on the coefficients, to compensate for the
bounds on the growth of the $T_n(x)$ as given by \eqref{bound_on_T_growth}, on the
fixed interval $[-a,a]$. 

We detail a bit more the proof of Fact (ii), since it is only
(slightly) more complex. Choose for example $b_k =
\frac{1}{2^k}$; since $f$ is entire, choose the Bernstein
ellipse of radius $3C$, on which $f$ is bounded by some given $B$;
as in the proof of Theorem \ref{thm:dfixedlinear},  this states that
the coefficients $c_n$ are bounded by
\begin{eqnarray} \label{bound_on_cs_2}
|c_n| &\leq& B' (3C)^{-n}~,
\end{eqnarray}
for some $B'$ independent of $n$.

As before, thanks to the expression \eqref{new_t}, we can bound the
growth of the modified Chebyshev polynomials on $[-C, C]$ by
\begin{eqnarray} \label{bound_on_T_growth}
\max_{x \in [-C, C]} |T_n(x/2)| & \leq & B'' C^n~,
\end{eqnarray}
for some $B''$ independent of $n$. 

With these choices for $\omega$ and $(b_k)_{k \in \mathbb{N}}$,
\eqref{coeff_growth} follows now from \eqref{bound_on_cs_2} and
\eqref{bound_on_T_growth}.
\end{proof}

We can now give our main result for the case when $d_n$ and $r_n$ both
grow. The essential difference from before is in the centering and in assumption (ii) below which stresses the dependence on the growth rate of the degree sequence. 
  
  \begin{thm}\label{thm:dgrowslinear} 
Assume the same setup as in Lemma \ref{lem:prelim_dgrows}, 
with the following additional constraints on the entire function $f$:
  
  \begin{enumerate}
  \item[(i)] Let $C:=C(1)$ be chosen according to Theorem
    \ref{thm:eigbound}. 
Let $f_k := \sum_{i=0}^k c_i \Phi_i$ denote the $k$th
truncation of this series on $[-C, C]$. 
Then 
  \[
  \sup_{0\le \abs{x} \le C} \abs{f(x) - f_k(x)} \le M \exp\left( -\alpha k h(k) \right), \quad \text{for some $\alpha >2$ and $M >0$},
  \]
 where $h$ has been defined in \eqref{eq:whatish}. 
 \item[(ii)]  Recall the definition of sequence $(r_n)$ from \eqref{eq:choosern} with a choice of $\beta < 1/\alpha$. 
Then $f$ and its sequence of truncations, $f_{r_n}$, satisfy
 \[
 \lim_{n\rightarrow \infty} \abs{f_{r_n}\left(2d_n (2d_n-1)^{-1/2} \right) - f\left(2d_n(2d_n-1)^{-1/2}\right)}=0.
 \]
  \end{enumerate}
   Define now the array of constants 
  \[                            
  m^f_k(n):= \sum_{i=1}^k \frac{c_i}{(2d_n-1)^{i/2}}\left( \mu_i(d_n) - 
  \mathbf{1}_{(\text{$i$ is even})}(2d_n-2)n \right). 
  \] 
 If conditions $(i)$ and $(ii)$ above are satisfied, the sequence of random variables
  \[
  \left(\sum_{i=1}^n f(\lambda_i) - nc_0 - m^f_{r_n}(n)\right)_{n\in \NN}
  \]
  converges in law to a normal random variable with mean zero and variance $\sigma_f^2 = \sum_{k=1}^\infty 2k c_k^2$.
  \end{thm}
  
  \begin{rmk} 
Note the significance of the term $h(k)$. The presence of $h(k)$,
which is usually a logarithmic term as in \eqref{eq:computehk},
demands somewhat more than just analyticity of $f$. Similarly,
requirement $(ii)$ requires convergence of the truncations sequence,
evaluated at points diverging to $\infty$; it is a kind of
``diagonal'' convergence, which is not automatically satisfied even
for entire functions.
\end{rmk}

\begin{proof}
The proof is almost identical to the proof of Theorem \ref{thm:dfixedlinear} and we only highlight the slight differences. As before, define
\[
nc_0+ \Y[f]{n}:= \sum_{k=1}^{r_n} c_k \Nc[n]{k} = \sum_{i=1}^n \left(\sum_{k=1}^{r_n} c_k \Phi_k(\lambda_i)\right) - m^f_{r_n}(n). 
\]

To prove that $\Y[f]{n}$ converges in law to $N(0, \sigma_f^2)$, we
use Fact (ii) from Lemma \ref{lem:prelim_dgrows} together with
assumption (ii); by Theorem \ref{thm:dgrowstight}, the convergence follows.
We only need to show that
$$\abs{\sum_{i=1}^{n}f(\lambda_i) - nc_0- \Y[f]{n}}$$ converges to zero in
probability. The convergence for $\lambda_1$ is given by assumption
(ii),  while the rest of it is assured by assumption (i) and Theorem \ref{thm:eigbound}.
\end{proof}  

\section{Appendix}

We will compute the exact expression for the number of cyclically reduced words of length $k$ on letters $\pi_1, \ldots, \pi_d, \pi_1^{-1}, \ldots, \pi_d^{-1};$ specifically, we will show
\begin{lemma}
\label{adk}
\begin{align*}
    a(d,2k) &= (2d-1)^{2k} -1+2d,
    & a(d,2k+1)&= (2d-1)^{2k+1}+1.
\end{align*}
\end{lemma}
\begin{proof}
This is a quick exercise in inclusion-exclusion.  The proof requires some notation, but this should not obscure the simplicity of the ideas.  Define
\[
\Pi_k = \left\{ \pi_1, \pi_2, \ldots, \pi_d, \pi_1^{-1}, \pi_2^{-2} \ldots, \pi_d^{-1}
\right\}^k
\]
to be all words of length $k$ in these letters.  Let $G = \ZZ / k\ZZ$ denote the cyclic group of order $k$, and for any subset $S \subseteq G,$ define
\[
V_S = \left\{
 w = w_0 w_1 \cdots w_{k-1} \in \Pi_k ~\middle\vert~ w_s = w_{s+1}^{-1}~ s \in S
\right\},
\]
where the addition is performed in $G.$  The essential observation is that
\[
|V_S| = \begin{cases}
(2d)^{k - |S|} & k > |S| \\
2d & k = |S|, k~\text{even} \\
0 & k = |S|, k~\text{odd}.
\end{cases}
\]
To see the formula for $k > |S|,$ note that each $w_i$ with $i \neq S$ can be chosen freely from the alphabet.  Moreover, once these are chosen, the word can be completed uniquely by the rules of $V_S.$  The $k = |S|$ formula follow as in these cases, the word must be a single letter that alternates with its inverse, and this is only possible if the length of the word is even.

Having established these formulae, we can compute $a(d,k)$ by inclusion-exclusion,
\[
a(d,k) = \sum_{S \subseteq G} (-1)^{|S|} |V_S|
= \sum_{l=0}^{|S|-1} {k \choose l} (-1)^{l} (2d)^{k-l} + \begin{cases}
2d & k~\text{even} \\
0 & k~\text{odd}.
\end{cases}
\]
Noting that this is nearly the binomial formula, the desired expressions follow.
\end{proof}

 \bibliographystyle{alpha}
 \bibliography{poi-approx}

\end{document}